\documentclass[a4paper,english,10pt]{article}
\usepackage[applemac]{inputenc}
\usepackage{inputenc}
\usepackage{babel}
\usepackage{amsfonts}
\usepackage{amsmath}
\usepackage{amssymb}
\usepackage{amsthm}
\usepackage{calligra,mathrsfs}
\usepackage{graphicx}
\usepackage[all]{xy}
\usepackage{textcomp}
\usepackage[usenames,dvipsnames]{color}
\usepackage{tikz-cd}

\usetikzlibrary{arrows,cd,decorations.markings,shapes.geometric,shapes}

\newcommand{\cop}{\mathsf{copy}}
\newcommand{\del}{\mathsf{del}}

\newcommand{\beq}{\begin{equation}}
\newcommand{\eeq}{\end{equation}}

\pgfdeclarelayer{edgelayer}
\pgfdeclarelayer{nodelayer}
\pgfsetlayers{edgelayer,nodelayer,main}
\tikzstyle{none}=[]
\tikzset{baseline=(current  bounding  box.center)}

\newcommand*{\isoarrow}[1]{\arrow[#1,"\rotatebox{90}{\(\sim\)}"
]}
\usepackage{enumitem} 
\usepackage{hyperref}
\hypersetup{
    colorlinks = true,
    urlcolor = blue,
    linkcolor = blue,
    citecolor = purple
    } 
     
\newtheorem{thm}{Theorem}[section]  

\newtheorem{lem}[thm]{Lemma}
\newtheorem{defi}[thm]{Definition}
\newtheorem{prop}[thm]{Proposition}
\newtheorem{es}[thm]{Example}
\newtheorem{rem}[thm]{Remark}

\DeclareMathOperator{\Homs}{\mathscr{H}\text{\kern -3pt {\calligra\large om}}\,}

\title{A category theory framework for Bayesian learning}
\date{}

\author{{Kotaro Kamiya} \& {John Welliaveetil} \thanks{Authors have equal contribution.} \\
SyntheticGestalt Ltd \\
\text{\{k.kamiya, j.welliaveetil\}@syntheticgestalt.com}
}

\begin{document}
 \maketitle

\begin{abstract}
Inspired by the foundational works in \cite{FSBackprop} and \cite{CGgradientlearning}, we 
introduce a categorical framework to formalize Bayesian inference and learning. The
two key ideas at play here are the notions of Bayesian inversions and the 
functor GL as 
constructed in 
\cite[\S 2.1]{CGgradientlearning}.
We find that Bayesian learning is the simplest case 
of the learning paradigm described in \cite{CGgradientlearning}. 
We then obtain categorical formulations 
of batch and sequential Bayes updates while also verifying that the two coincide in 
a specific example.
\end{abstract}

{\hypersetup {linkcolor = black} 
\tableofcontents
}

\section{Introduction}

      A standard problem in Machine learning is to understand the relationship between 
      two variables or random vectors. Suppose 
      $\mathbf{x}$ is a random vector and $y$ is a
      random variable dependent on $\mathbf{x}$. 
      The additive model assumes that 
     there exists a function $f$ such that 
     $$y = f(\mathbf{x}) + \epsilon$$ 
     where $\epsilon$ is a random variable with mean $0$. 
     Our goal then reduces to estimating the function $f$.
      To this end, we consider a 
     parametrized family of functions $\{f_{\theta}\}_{\theta \in P}$ where $P$ is the parameter space
      and by means of
     a learning algorithm and a training data set, we 
     traverse $P$ to 
     find a candidate that we believe will provide a reasonable approximation 
     to the true function $f$. Several theoretical results underpin the validity 
     of this approach - for instance the Universal Approximation Theorem.      
     
      A formulation of the above framework within Category theory was achieved in the
     foundational work of 
     Fong, Spivak and Tuyeras \cite{FSBackprop}. In this paper, the authors 
     introduce the category $\mathbf{Learn}$ 
     whose objects are sets and a morphism from $A$ to $B$
   in $\mathbf{Learn}$ consists of 
   the following data. 
   A parameter set $M$, an implementation $I \colon M \times A \to B$, 
   an update function $u \colon M \times A \times B \to M$ and 
   a request function $r \colon M \times A \times B \to A$. 
   In related work, the authors of \cite{CGgradientlearning},
   develop a more general theory by introducing the notion 
   of a Cartesian Reverse Differential category (cf. \cite[\S 2.2]{CGgradientlearning}).
    Associated 
   to a CRDC - $\mathcal{C}$, they construct  
   the analog of the category $\mathbf{Learn}$ 
   as a composite of the $\mathbf{Para}$ and $\mathbf{Lens}$ 
   constructions (cf. \cite[\S 2.1, 2.3]{CGgradientlearning}). In this setting, the learning algorithm will be a functor 
   of type $\mathbf{Para}(\mathcal{C}) \to \mathbf{Para}(\mathbf{Lens}(\mathcal{C}))$.
  Our goal in this paper is to employ both these approaches to 
  develop a categorical framework to discuss Bayesian learning. 
  
  While the approach outlined above to model the relationship between two random variables 
  is effective, in real world situations, it is often the case that there is no such function $f$.
  Given the noisy nature of the data, it is more reasonable to model 
  the conditional probability $p(y|x)$ i.e. to ask the question - 'Given $x$, what is the probability 
  that we have $y$ ?'. 
  
  As before, we model the conditional probability using a parametrized family of distributions -
 $p(y | x ; \theta)$. However in contrast to 
 our approach from before, we assume that $\theta$ is a random variable and 
 strive to obtain a distribution over $\theta$ that agrees with the given training data. 
 More precisely, we choose a prior distribution  
 $q(\theta)$ on $\theta$ and update this prior to 
 obtain the posterior distribution on $\theta$. 
 As opposed to
 gradient based approaches to learning, Bayesian machine learning updates the 
 prior distribution on $\theta$
  by exploiting Bayes Theorem. The posterior distribution is defined 
  upto normalizing constant by the formula 
  \[ p(\theta | y,x ) \propto p(y | x,\theta) q(\theta|x). \]
 There is considerable flexibility in this approach in that we can use the posterior 
 to obtain point estimates of the parameter $\theta$ via the MAP estimate and perform inference 
 by integrating over the entire distribution. We continue this discussion
  in greater detail in \S \ref{section : bayesian inference prelims} where
  we illustrate the essential features of Bayesian learning in the classical context using a simple 
 example.
  
     Our goal in this paper is to introduce a framework within category theory that allows us 
     to formalize this set up. We draw on the theory already developed in 
     \cite{TFSyntheticApproach}, \cite{Sm2020bayesian} and 
     \cite{CJ2019disintegration}.
     While we proceed as in 
     \cite{CGgradientlearning} by 
     associating to a parametrized function 
     a morphism in a generalized lens category 
     which allows for the backward transmission 
     of information, 
     an important observation 
     is that the Bayes Learning framework 
     simplifies the situation drastically. 
     This is essentially due to the 
     fact that 
     with the correct setup, Bayesian inversion 
     is a well defined dagger functor and
     the functor $\mathcal{C}^{\mathrm{op}} \to \mathrm{Cat}$ responsible 
     for defining the generalized lens breaks down. 
     We interpret this to mean that Bayes Learning is 
     the simplest form of learning that adheres to the framework 
     discussed in \cite{CGgradientlearning}. 
     See also Remark \ref{rem : bayesian learning easiest}.
  

   The work of Cho and Jacobs in \cite{CJ2019disintegration} 
   and Fritz in \cite{TFSyntheticApproach} introduces the notion 
   of Markov categories as a suitable 
   framework within which one can discuss ideas from Probability theory 
   such as Bayesian inversion, disintegration, jointification, conditionalization...
   Hence, let $\mathcal{C}$ be a Markov category. 
   We use 
   the categories FinStoch and Stoch (cf. Example \ref{es : markov category examples})
  as guiding examples for the theory we develop.
  Roughly speaking, the category FinStoch has finite sets as objects 
  and Markov kernels as morphisms between them while the objects of
  Stoch are measurable spaces and 
  morphisms $f \colon X \to Y$ imply for every 
  $x \in X$, a probability
  distribution $p(\cdot|x)$ on $Y$.      
   Central to this paper is the notion of Bayesian inversion
   which in the context of FinStoch with reference to a 
   morphism $p \colon X \to Y$ 
   corresponds to calculating the conditional 
   $p(X|y)$ for $y \in Y$ using Bayes theorem to obtain a morphism 
   $p^\dag \colon Y \to X$. 
   Note that the inversion which is a map 
   $Y \to X$ is not necessarily unique 
   and in the case of FinStoch does not even always exist.
   To get around this issue, firstly we restrict our attention 
   to those Markov categories $\mathcal{C}$ which always admit 
   Bayesian inversion.   
   Secondly, we make use of the symmetric monoidal 
   category $\mathrm{ProbStoch}(\mathcal{C})$
   which was introduced first by Cho and Jacobs in 
   \cite[\S 5]{CJ2019disintegration} as the category 
   of couplings or equivalently the category whose objects coincide with those of
   the slice category $I \downarrow \mathcal{C}$ 
   and morphisms 
   are obtained by quotienting with respect 
   to the relation of \emph{almost sure equality}.
   Henceforth, we write $\mathrm{PS}(\mathcal{C})$
   in place of $\mathrm{ProbStoch}(\mathcal{C})$ for 
   ease of notation.
   As a consequence of our assumption on $\mathcal{C}$, 
   Bayesian inversion is a well defined dagger functor on 
   $\mathrm{PS}(\mathcal{C})$.
   A more detailed explanation 
   can be found in \S \ref{section : bayesian inversions prelims}.
   
    We can interpret the set-up in \S \ref{section : bayesian inference prelims}, 
    as a parametrized function in a suitable Kleisli category. 
    In general, we assume that the category 
    $\mathcal{C}$ is equipped with the structure of an $\mathcal{M}$-actegory
    where $\mathcal{M}$ is symmetric monoidal. 
    We would like to make use of the $\mathbf{Para}_{\mathcal{M}}$ construction 
    from \cite{CGHR2021foundations} which generalizes
    \cite{FSBackprop}. 
    For reasons explained above, we work with the category 
    $\mathrm{PS}(\mathcal{C})$ and in
    Lemma \ref{lem : PS(M) actegories}, we show that 
    $\mathrm{PS}(\mathcal{C})$ 
    has the structure of an $\mathrm{PS}(\mathcal{M})$-actegory
    provided we impose certain technical assumptions 
    on $\mathcal{M}$ and its action on $\mathcal{C}$. 
    We simplify notation henceforth and write 
    $\mathrm{PS}$ in place of $\mathrm{PS}$.  
    
    Our definition of the functor BayesLearn in Section \S \ref{section : bayes learn} is 
    inspired in large part by \cite{CGgradientlearning}. 
    The key to the construction of a gradient based learner in \cite{CGgradientlearning}
    is a functor 
    \[ \mathrm{GL} \colon \mathbf{Para}(\mathcal{E}) 
    \to \mathbf{Para}(\mathrm{Lens}(\mathcal{E})) \]
    where we require that $\mathcal{E}$ is a Cartesian 
  reverse differential category. To fully describe the learning mechanism, the authors
  introduce update and error endofunctors on $\mathbf{Para}(\mathcal{E})$. 
  These three pieces when employed in unison and 
  for appropriate choices of the update and displacement maps, can 
  describe a wide range of optimization algorithms (\cite[\S 3.4]{CGgradientlearning}).
  
  When adapting the framework of 
  \cite{CGgradientlearning}
   to describe Bayes learning, we see immediately 
  that we must work with a more general notion of Lens as in 
  \cite{Sm2020bayesian}. To this end, we 
  proceed as in 
  \cite[Definition 3.1]{Sm2020bayesian} by introducing a functor 
  \[\mathrm{Stat} \colon \mathrm{PS}(\mathcal{C})^{\mathrm{op}} \to \mathrm{Cat}\]
   with which we define the associated Grothendieck lens 
   $\mathrm{Lens}_{\mathrm{Stat}}$.
  As Bayesian inversion is a well defined dagger functor, 
  we deduce the existence of a well defined 
  functor 
  \[R \colon \mathrm{PS}(\mathcal{C}) \to \mathrm{Lens}_{\mathrm{Stat}}.\]
  It follows that we have a functor
  \[ \mathbf{Para}_{\mathrm{PS}(\mathcal{M})}(R) 
  \colon \mathbf{Para}_{\mathrm{PS}(\mathcal{M})}(\mathrm{PS}(\mathcal{C})) \to \mathbf{Para}_{\mathrm{PS}(\mathcal{M})}(\mathrm{Lens}_{\mathrm{Stat}}).\]
  We refer to the functor 
  $\mathbf{Para}_{\mathrm{PS}(\mathcal{M})}(R)$ as the BayesLearn functor. 
  It captures the essential features of the Bayes learning algorithm. 
  Indeed, given a parametrized morphism $X \to Y$
  with a state on $X$, after 
  choosing a prior state on the parameter object $P$, 
  we obtain a morphism in 
  $\mathrm{PS}(\mathcal{C})$ i.e. 
  a morphism 
  \[f \colon (M,\pi_M) \odot (X,\pi_X) \to (Y,\pi_Y).\]
  By construction, 
  the induced morphism 
  \[\mathrm{BayesLearn}(f) \colon ((M \odot X,\pi_M \odot \pi_X),(M \odot X,\pi_M \odot \pi_X)) \to ((Y,\pi_Y),(Y,\pi_Y)) \]
  is given by a 
  pair of morphisms 
  \[f \colon (M \odot X,\pi_M \odot \pi_X) \to (Y,\pi_Y)\]
   in the forward direction and 
   \[f^{\dag} \colon (Y,\pi_Y) \to (M \odot X,\pi_M \odot \pi_X)\]
   in the backwards direction by Bayesian inversion. 
   This
   provides a rough description of the
    Bayesian learning procedure which we
   elaborate on in \S \ref{section : bayes learn} and
  \S \ref{sec : bayes learning algo} where we also 
   give a formulation for the Bayes predictive density.  
   
  The classical formulation of Bayesian learning enables one to update a 
  given prior distribution on a parameter space using a training data set and Bayesian inversion. 
  In \S \ref{section : bayes updates}, we provide a category theoretic formulation 
  of this phenomenon. Central to the discussion will be the notion 
  of a training set. 
  To capture the notion of a single training instance, we restrict 
  our attention to those Markov categories which 
  are of the form $\mathrm{Kl}(\mathcal{P})$ where 
  $\mathcal{P} \colon \mathcal{D} \to \mathcal{D}$ is a symmetric monoidal 
  monad on a symmetric monoidal category $\mathcal{D}$. 
  In this context, we introduce the notion 
  of an elementary point of an object $X$ in $\mathcal{C}$ 
  (cf. Definition \ref{def : elementary point}). In the case of 
  FinStoch or Stoch, these correspond to states on 
  $X$ which concentrate at a point $x \in X$. 
  We show that there are two ways one can 
  update the prior. 
  As before, let us suppose we have a model i.e. 
  a morphism $f \colon M \odot X \to Y$ in 
  $\mathcal{C}$. 
  In addition we are given a prior distribution $\pi_{M,0}$ on $M$ and a state 
  $\pi_X \colon I \to X$. Via Bayesian inversion with respect 
  to $\pi_{M,0} \odot \pi_X$ and after conditionalizing, we get a channel 
  $f^{\dag}_{\mathrm{joint},0} \colon X \otimes Y \to M$. The precise details can be found in 
  \S \ref{section : bayes learn}.
 Let 
  \[\mathbf{T} := [(x_1 \otimes y_1), \ldots, (x_n \otimes y_n)] \]
  be a list of \emph{elementary points} (cf. Definition \ref{def : elementary point})
  belonging to $X \otimes Y$.
  To obtain the posterior,
  we can proceed sequentially i.e. we update the prior 
  $\pi_{M,0}$ by the composition 
  $I \xrightarrow{\delta_{x_1} \otimes \delta_{y_1}} X \otimes Y \to M$ to get a
  state $\pi_{M,1}$ on $M$.
  Note that $x_1$ corresponds to a morphism 
  $I_\mathcal{D} \to X$ and $\delta_{x_1}$ is the image 
  of this map in $\mathcal{C}$. 
  We then implement the inversion procedure with 
  respect to $\pi_{M,1}$
  to get an updated channel $X \otimes Y \to M$ and continue as before. 
  Repeating the above step to run through the training data set 
  will end with a state $\pi_{M,n}$ on $M$ which 
  in the case of FinStoch defines a distribution 
  on $M$ which we call the posterior.
  
  In a similar fashion, we can also define batch updates 
  associated to the training set $\mathbf{T}$. 
  To do so, we work with the space $Z := X \otimes Y$ and 
  set $Z_n := \otimes^n Z$. A channel $f \colon M \to Z$ naturally 
  gives us a channel $M \to Z_n$ by the composition 
  \[M \xrightarrow{\mathbf{copy}^n_M} \otimes^n M \xrightarrow{\otimes^n f} Z_n.\]
  Here we abuse notation and 
  write 
  $\mathbf{copy}^n_M$ for the 
  composition
  \[M \xrightarrow{\mathbf{copy}_M} M \otimes M
  \xrightarrow{\mathrm{id} \otimes \mathbf{copy}_M} M \otimes M \otimes M \ldots \xrightarrow{\mathrm{id} \otimes \ldots \otimes \mathbf{copy}_M} \otimes^n M.\]
  Note that the map $M \to Z$ is obtained naturally from 
  the model $f \colon M \odot X \to Y$ via the composition
  \[M \odot I \xrightarrow{\mathrm{id} \odot \pi_X} M \otimes X \xrightarrow{f \otimes \mathrm{id}} 
  Y \otimes X \simeq X \otimes Y.\] 
 The Bayesian inversion with respect to $\pi_{M,0}$ 
 implies a channel $Z_n \to M$. 
The composition 
\[ I \xrightarrow{(\delta_{x_1} \otimes \delta_{y_1}) \otimes \ldots \otimes (\delta_{x_n} \otimes \delta_{y_n})} Z_n \to M \] 
 defines a state on $M$ which we refer to as the batch update 
 with respect to $\mathbf{T}$ and denote $\pi_{M,\mathbf{T}}$.
 
 It is natural in this setting to ask if there is a relationship between the sequential and batch updates
 or what hypothesis we must impose on $\mathcal{C}$ to ensure that they are equal.
 We leave this question open to further investigation and show that
 the two are equal in the case of FinStoch
  (cf. Examples \ref{es : batch updates} and \ref{es : sequential updates}).
\\ 

\noindent \textbf{String diagrams:} In this paper, we orient our string diagrams
from top to bottom and display monoidal product as moving from left to right. 
All string diagrams have been made using DisCoPy \cite{discopy}. 
 
\section{Preliminaries}

\subsection{Bayesian Inference} \label{section : bayesian inference prelims}

    Our goal in this section is to give a high level overview of the 
 approach Bayesian learning takes in modelling
  the relationship between random variables of interest. 
 We use the lecture \cite{Tbayesinf} as a reference to the material in this section. 
 
   Let us illustrate the approach with an example. 
   Let $\mathbf{x} = (x_1,\ldots,x_n)$
    be a random vector and $y$ be a random variable. 
   We assume that $y$ is related to $\mathbf{x}$ via the equation 
   \[ y = f(\mathbf{x}) + \epsilon \] 
   where $\epsilon$ is normally distributed with mean $0$ and variance $\sigma^2$ and 
   $f(\mathbf{x}) := \beta^T \mathbf{x}$ where $\beta \in \mathbb{R}^n$. 
   For simplicity, we assume that the quantity $\sigma$ is a known constant. 
   
   Let us suppose we are given a training set
    $\mathbf{T} := \{(\mathbf{x}_1,y_1),\ldots,(\mathbf{x}_N,y_N)\}$. 
    One can divide the Bayesian approach into two parts. 
 As outlined in the introduction we 
 treat $\beta$ as a random variable and
 choose a prior distribution  
 which encodes our pre-conceived beliefs about $\beta$. 
 Then using the set $\mathbf{T}$, we update the 
 distribution on $\beta$ to get its posterior distribution by applying
 Bayes' rule. 
 The posterior distribution can then be used to obtain 
 point estimates of $\beta$. 
 
  Let us assume an improper prior on $\beta$ i.e. $q(\beta) = 1$. 
  Applying Bayes' rule gives 
  \[ p(\beta | \mathbf{T}) \propto p(\mathbf{T} | \beta) q(\beta) \]
  Observe that our assumption on the nature of the underlying data implies 
  that $p(y|\mathbf{x})$ is normally distributed with mean $\beta^T \mathbf{x}$
  and variance $\sigma^2$. Assuming that the training data is independent, we get 
  \[p(\mathbf{T}|\beta) \propto \prod_i p(y_i| \beta, \mathbf{x}_i) q(\beta). \] 
  We assumed above that the distribution on $\mathbf{x}$ is known and hence in effect 
  we condition every random variable with respect to $\mathbf{x}$. 
  Hence, 
  \[p(\mathbf{T} | \beta) =  \frac{1}{({2\pi})^\frac{N}{2}\sigma^N } \mathrm{exp}(-\frac{\sum_i (y_i - \beta^T \mathbf{x}_i)^2}{2\sigma^2})\]
  and 
  \[p(\beta | \mathbf{T}) =  \frac{1}{ ({2\pi})^\frac{N}{2}\sigma^N} \mathrm{exp}(-\frac{\sum_i (y_i - \beta^T \mathbf{x}_i)^2}{2\sigma^2}) \]
  
   
   In this manner, we have updated the prior distribution on $\beta$ in accordance with the given data. 
  We now sketch how one might perform inference. Let $\mathbf{x}_*$ be a
  data point that does not necessarily belong to the training set $\mathbf{T}$. 
   We would like to find $p(y_*| \mathbf{x}_*, \mathbf{T})$. 
  We outline two ways on how to proceed. \\
  
  \begin{enumerate}
  \item Let \[ \beta_{\mathrm{MAP}} := \mathrm{argmax}_\beta p(\beta | \mathbf{T}) \]
    This is called the \emph{maximum a posteriori} estimate of the parameter $\beta$. 
    It represents a single best guess for the parameter given the training data. 
    We then set 
    \[ p(y_*| \mathbf{x}_*, \mathbf{T}) := p(y_* | \mathbf{x}_*, \beta_{\mathrm{MAP}}).\]
     Observe that in the context of our ongoing example, 
     $\beta_{\mathrm{MAP}}$ coincides with $\beta_{\mathrm{OLS}}$ 
    - the ordinary least squares estimate. \\
 
  \item While the MAP estimate of $\beta$ represents a suitable guess for the parameter, 
  when performing inference we can do better by leveraging our knowledge of the entire 
  posterior distribution. 
  The {true Bayesian} way is to integrate over the posterior distribution $\beta$ i.e.
  \[ p(y_*| \mathbf{x}_*, \mathbf{T}) 
  := \int p(y_* | \mathbf{x}_*, \beta) p(\beta | \mathbf{T} ) d\beta.\]
  In practice, owing to the intractability of the integral above due to the fact that 
  the posterior distribution is not likely analytic a possible solution is
  to proceed by employing a suitable approximation strategy. 
  See for instance \cite{approxintegralsevans} or \cite{variationalvectormachines}.
  \end{enumerate} 
   
\begin{rem}
 \emph{Observe that the posterior distribution on $\beta$ with regards 
 to the training set $\mathbf{T}$ can also be obtained 
 by sequential updates. This is due to the following phenomenon. 
 Suppose $\mathbf{T}$ can be split into two training sets 
 $\mathbf{T}_1$ and $\mathbf{T}_2$ where
 $\mathbf{T}_1 := \{(\mathbf{x}_1,y_1),\ldots,(\mathbf{x}_r,y_r)\}$ 
 and 
 $\mathbf{T}_2 := \{(\mathbf{x}_{r+1},y_{r+1}),\ldots,(\mathbf{x}_N,y_N)\}$. 
Let $q_1$ be the posterior distribution 
$p(\beta | \mathbf{T}_1)$ on $\beta$ obtained by updating the prior 
$q$.
Likewise, let $q_2$ be the posterior distribution
 on $\beta$ obtained by updating the prior 
$q_1$ using the dataset $\mathbf{T}_2$.
A simple calculation shows that 
the distribution $q_2$ coincides with the posterior 
obtained by updating $q$ using the entire 
dataset $\mathbf{T}$.
In \S \ref{section : bayes updates}, 
we formulate a categorical version 
of the above observation.}
\end{rem} 

\subsection{Bayesian inversions and $\mathrm{PS}$} \label{section : bayesian inversions prelims}

    The theory of Markov categories provides a categorical framework within which 
    we can discuss and formalize notions from Probability theory. We use the notation and 
    conventions introduced in \cite{TFSyntheticApproach}. 
    Note that this notion of Markov categories coincides with that of affine CD-categories 
    as introduced by Cho and Jacobs in \cite{CJ2019disintegration}. 
 
 \begin{defi}
  \emph{A} Markov category \emph{is a semicartesian symmetric monoidal category} $(\mathcal{C},I,\otimes)$ 
  \emph{in which every object X is equipped with the structure of a commutative internal comonoid.
 We denote the comultiplication and 
 counit maps by}
  \[ \mathbf{copy}_X \colon X \to X \otimes X\] 
   \emph{and}
   \[ \mathbf{del}_X \colon X \to I \]
   \emph{respectively and require that they satisfy 
   certain natural conditions (cf. \cite[Definition 2.1]{TFSyntheticApproach}).}
   
        
 \end{defi}
 
The definition above suggest that one can think of
  a morphism $\pi_X \colon I \to X$ in 
  a Markov category $\mathcal{C}$ as a probability distribution on $X$. We refer 
  to $\pi_X$ as defining a \emph{state} on $X$. 
 Likewise, a morphism $f \colon X \to Y$ in $\mathcal{C}$ will be 
 called a \emph{channel}. 

 For the remainder of this section we fix a Markov category $\mathcal{C}$. 
 
 \begin{es} \label{es : markov category examples}
 \emph{We introduce two examples which we will continue to make reference to throughout the course of the
 paper.}
 \begin{enumerate}
 \item \emph{We define the category $\mathrm{FinStoch}$ to be the category 
 of finite sets with channels. A morphism $f \colon X \to Y$ 
 in $\mathrm{FinStoch}$ is given by a morphism of sets 
 $f \colon X \to \mathrm{Dist}(Y)$ where 
 $\mathrm{Dist}(Y)$ is the set of probability distributions on $Y$
 i.e.}
 \[ \mathrm{Dist}(Y) = \Big\{ \omega \colon Y \to [0,1] | \sum_{y \in Y} \omega(y) = 1 \Big\}. \]
 
 \emph{Given a set $X$, we define 
 $\mathbf{copy}_X$ to be the map that sends an element 
 $x \in X$ to the distribution on $X \times X$ which takes the value 
 $1$ at $(x,x)$ and $0$ everywhere else. 
 The terminal object $I$ is the set with a single element i.e. $I = \{*\}$
 and as a consequence we can identify $\mathrm{Dist}(I)$ 
 with
 $I$. 
 We define $\mathbf{del}_X$ to be the map 
that sends every element 
  $x \in X$ to $*$. 
 Lastly, given morphisms $f \colon X \to Y$ and $g \colon Y \to Z$, we 
 define the composition $f ; g \colon X \to Z$ to be such that 
 for every $x \in X$,
 \[(f ; g)(x) (z)  := \sum_{y \in Y} g(y)(z)f(x)(y). \] }
  \emph{Observe that $\mathrm{Dist}$ defines an endofunctor on the category of Finite sets. 
 One can check without difficulty that it is in fact a commutative monad. It follows that 
 $\mathrm{FinStoch}$ coincides with the Kleisli category $\mathrm{Kl}(\mathrm{Dist})$.}
 \item \emph{To deal with distributions 
 whose supports are not necessarily finite, 
 we introduce $\mathrm{Stoch}$ whose objects are 
 measurable spaces i.e. tuples of the form $(X, \Sigma_X)$ where $X$ is a set 
 and $\Sigma_X$ is a well defined $\sigma$-algebra on $X$. 
 Recall that the category of measurable spaces is endowed with a natural 
 symmetric monoidal structure 
 given by 
 \[(X,\Sigma_X) \otimes (Y,\Sigma_Y) := (X \times Y, \Sigma_X \otimes \Sigma_Y)\]
 where $\Sigma_X \otimes \Sigma_Y$ corresponds to the $\sigma$-algebra generated 
 by subsets of the form $U \times V$ where $U \in \Sigma_X$ and 
 $V \in \Sigma_Y$.}
 
 \emph{A morphism $(X,\Sigma_X) \to (Y,\Sigma_Y)$
  in $\mathrm{Stoch}$ is given by a map 
 \[f \colon X \times \Sigma_Y \to [0,1] \] such that 
 for every $S \in \Sigma_X$, the map $f(S,\_) \colon X \to [0,1]$ is measurable.
 Furthermore, we ask that for every $x$, $f(\_,x) \colon \Sigma_Y \to [0,1]$ is 
 a well defined probability measure. One checks that 
 $\mathrm{Stoch}$ is a symmetric monoidal category
 whose monoidal structure is inherited from the category of measurable spaces. 
 The unit object $I$ is the pair $(\{*\},\{\{*\},\emptyset\})$. 
 }
 
 \emph{As before, we must define copy and delete morphisms as well as specify how to compose 
 morphisms in Stoch. We do not define $\mathbf{del}$ as this is obvious from 
 the definition of $I$. 
 Given $(X,\Sigma_X) \in \mathrm{Ob}(\mathrm{Stoch})$, we define 
 \[ \mathbf{copy}_X \colon (X \times X) \times \Sigma_{X \times X}  \to [0,1] \]
 as the map that sends a pair $(S,x)$ to $1$ if $(x,x) \in S$ and 
 $0$ otherwise. Lastly, given morphisms
 $f \colon (X,\Sigma_X) \to (Y,\Sigma_Y)$ and 
 $g \colon (Y,\Sigma_Y) \to (Z,\Sigma_Z)$, 
 we define 
 $f ; g$ 
 to be the map 
 $X \times \Sigma_Z \to [0,1]$ to be given by 
 \[ (S,x) \mapsto \int g(S, y) f(dy,x)  \] 
 where we abuse notation and write $f(x)$ for the 
 measure $\Sigma_Y \to [0,1]$ given by $T \mapsto f(T,x)$.}
 \end{enumerate} 
 \end{es}

 \begin{defi} \label{def : disintegration and jointification} 
   \emph{Let $\pi_{X \otimes Y} \colon I \to X \otimes Y$ be a joint state. A} disintegration 
   \emph{of $\pi_{X \otimes Y}$ will be a pair consisting of a channel 
   $f \colon X \to Y$ and a state $\pi_X \colon I \to X$ such that 
   the following string diagrams are equal.}
   \beq
		\label{cond_dist}
		\begin{tikzpicture}[baseline=(0.base)]
\begin{pgfonlayer}{nodelayer}
\node (0) at (0, 1.5) {};
\node [] (1) at (0.0, 1.25) {};
\node [] (2) at (0.0, 0.0) {};
\node [style=none, fill=white, right, scale=0.7] (3) at (0.1, 1.0) {$X$};
\node [] (4) at (1.0, 1.25) {};
\node [] (5) at (1.0, 0.0) {};
\node [style=none, fill=white, right, scale=0.7] (6) at (1.1, 1.0) {$Y$};
\node [] (7) at (-0.25, 1.25) {};
\node [] (8) at (1.25, 1.25) {};
\node [] (9) at (1.25, 1.75) {};
\node [] (10) at (-0.25, 1.75) {};
\node [style=none, fill=white, scale=0.7] (11) at (0.5, 1.5) {$\pi_{X\otimes Y}$};
\node [style=none, fill=white] (12) at (2.25, 1.5) {=};
\node [] (13) at (3.75, 2.25) {};
\node [] (14) at (3.75, 1.75) {};
\node [style=none, fill=white, right, scale=0.7] (15) at (3.85, 2.0) {$X$};
\node [] (16) at (3.75, 1.5) {};
\node [] (17) at (3.25, 1.25) {};
\node [] (18) at (4.25, 1.25) {};
\node [] (19) at (3.25, 0.0) {};
\node [style=none, fill=white, right, scale=0.7] (20) at (3.35, 1.0) {$X$};
\node [] (21) at (4.25, 0.75) {};
\node [style=none, fill=white, right, scale=0.7] (22) at (4.35, 1.0) {$X$};
\node [] (23) at (4.25, 0.25) {};
\node [] (24) at (4.25, 0.0) {};
\node [style=none, fill=white, right, scale=0.7] (25) at (4.35, 0.0) {$Y$};
\node [circle, fill=black, scale=0.485] (26) at (3.75, 1.5) {};
\node [] (26) at (3.5, 2.25) {};
\node [] (27) at (4.0, 2.25) {};
\node [] (28) at (4.0, 2.75) {};
\node [] (29) at (3.5, 2.75) {};
\node [style=none, fill=white, scale=0.7] (30) at (3.75, 2.5) {$\pi_{X}$};
\node [] (31) at (4.0, 0.25) {};
\node [] (32) at (4.5, 0.25) {};
\node [] (33) at (4.5, 0.75) {};
\node [] (34) at (4.0, 0.75) {};
\node [style=none, fill=white, scale=0.7] (35) at (4.25, 0.5) {$f$};
\end{pgfonlayer}
\begin{pgfonlayer}{edgelayer}
\draw [in=90, out=-90] (1.center) to (2.center);
\draw [in=90, out=-90] (4.center) to (5.center);
\draw [-, fill={white}] (7.center) to (8.center) to (9.center) to (10.center) to (7.center);
\draw [in=90, out=-90] (13.center) to (14.center);
\draw [in=90, out=180, looseness=0.9391] (16.center) to (17.center);
\draw [in=90, out=0, looseness=0.9391] (16.center) to (18.center);
\draw [in=90, out=-90] (14.center) to (16.center);
\draw [in=90, out=-90] (17.center) to (19.center);
\draw [in=90, out=-90] (18.center) to (21.center);
\draw [in=90, out=-90] (23.center) to (24.center);
\draw [-, fill={white}] (26.center) to (27.center) to (28.center) to (29.center) to (26.center);
\draw [-, fill={white}] (31.center) to (32.center) to (33.center) to (34.center) to (31.center);
\end{pgfonlayer}
\end{tikzpicture}
	\eeq
   \emph{If every joint state admits 
   a disintegration then we say that the category $\mathcal{C}$ admits}
   conditional distributions.
   
   \emph{Likewise, 
  given a pair $f \colon X \to Y$ and $\psi \colon I \to X$, 
  we can easily define a state on $X \times Y$ via the following string diagram.}
  $$
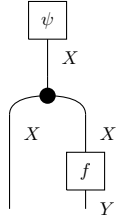
\begin{tikzpicture}[baseline=(0.base)]
    \begin{pgfonlayer}{nodelayer}
    \node (0) at (0, 1.5) {};
    \node [] (1) at (0.5, 2.25) {};
    \node [] (2) at (0.5, 1.75) {};
    \node [style=none, fill=white, right, scale=0.7] (3) at (0.6, 2.0) {$X$};
    \node [] (4) at (0.5, 1.5) {};
    \node [] (5) at (0.0, 1.25) {};
    \node [] (6) at (1.0, 1.25) {};
    \node [] (7) at (0.0, 0) {};
    \node [style=none, fill=white, right, scale=0.7] (8) at (0.1, 1.0) {$X$};
    \node [] (9) at (1.0, 0.75) {};
    \node [style=none, fill=white, right, scale=0.7] (10) at (1.1, 1.0) {$X$};
    \node [] (11) at (1.0, 0.25) {};
    \node [] (12) at (1.0, 0) {};
    \node [style=none, fill=white, right, scale=0.7] (13) at (1.1, 0.0) {$Y$};
    \node [circle, fill=black, scale=0.632] (14) at (0.5, 1.5) {};
    \node [] (14) at (0.25, 2.25) {};
    \node [] (15) at (0.75, 2.25) {};
    \node [] (16) at (0.75, 2.75) {};
    \node [] (17) at (0.25, 2.75) {};
    \node [style=none, fill=white, scale=0.7] (18) at (0.5, 2.5) {$\psi$};
    \node [] (19) at (0.75, 0.25) {};
    \node [] (20) at (1.25, 0.25) {};
    \node [] (21) at (1.25, 0.75) {};
    \node [] (22) at (0.75, 0.75) {};
    \node [style=none, fill=white, scale=0.7] (23) at (1.0, 0.5) {$f$};
    \end{pgfonlayer}
    \begin{pgfonlayer}{edgelayer}
    \draw [in=90, out=-90] (1.center) to (2.center);
    \draw [in=90, out=180, looseness=0.9391] (4.center) to (5.center);
    \draw [in=90, out=0, looseness=0.9391] (4.center) to (6.center);
    \draw [in=90, out=-90] (2.center) to (4.center);
    \draw [in=90, out=-90] (5.center) to (7.center);
    \draw [in=90, out=-90] (6.center) to (9.center);
    \draw [in=90, out=-90] (11.center) to (12.center);
    \draw [-, fill={white}] (14.center) to (15.center) to (16.center) to (17.center) to (14.center);
    \draw [-, fill={white}] (19.center) to (20.center) to (21.center) to (22.center) to (19.center);
    \end{pgfonlayer}
\end{tikzpicture}$$
  \emph{We refer to this as the jointification 
  of $f$ and $\psi$.}
 \end{defi} 
 
 We will require a more general version of conditionalization which permits 
 us to consider joint distributions parametrized by another object. 

 \begin{defi}
 \emph{Let $\mathcal{C}$
  be a Markov category. We say that $\mathcal{C}$
   has} conditionals \emph{if for every morphism 
   $s \colon A \to X \otimes Y$, there is 
   $t \colon X \otimes A \to Y$ such that 
   we have the following equality of string diagrams.}
   \beq
		\label{cond_eq}
		\begin{tikzpicture}[baseline=(0.base)]
\begin{pgfonlayer}{nodelayer}
\node (0) at (0, 2.5) {};
\node [] (1) at (0.5, 5.0) {};
\node [] (2) at (0.5, 2.75) {};
\node [style=none, fill=white, right, scale=0.7] (3) at (0.6, 5.0) {$A$};
\node [] (4) at (0.0, 2.25) {};
\node [] (5) at (0.0, 0.0) {};
\node [style=none, fill=white, right, scale=0.7] (6) at (0.1, 2.0) {$X$};
\node [] (7) at (1.0, 2.25) {};
\node [] (8) at (1.0, 0.0) {};
\node [style=none, fill=white, right, scale=0.7] (9) at (1.1, 2.0) {$Y$};
\node [] (10) at (-0.25, 2.25) {};
\node [] (11) at (1.25, 2.25) {};
\node [] (12) at (1.25, 2.75) {};
\node [] (13) at (-0.25, 2.75) {};
\node [style=none, fill=white, scale=0.7] (14) at (0.5, 2.5) {$s$};
\node [style=none, fill=white] (15) at (2.25, 2.5) {=};
\node [] (16) at (5.25, 5.0) {};
\node [] (17) at (5.25, 4.75) {};
\node [style=none, fill=white, right, scale=0.7] (18) at (5.35, 5.0) {$A$};
\node [] (19) at (5.25, 4.5) {};
\node [] (20) at (4.25, 4.25) {};
\node [] (21) at (6.25, 4.25) {};
\node [] (22) at (4.25, 3.75) {};
\node [style=none, fill=white, right, scale=0.7] (23) at (4.35, 4.0) {$A$};
\node [] (24) at (6.25, 0.75) {};
\node [style=none, fill=white, right, scale=0.7] (25) at (6.35, 4.0) {$A$};
\node [] (26) at (3.75, 3.25) {};
\node [] (27) at (3.75, 2.75) {};
\node [style=none, fill=white, right, scale=0.7] (28) at (3.85, 3.0) {$X$};
\node [] (29) at (5.25, 3.25) {};
\node [] (30) at (5.25, 1.75) {};
\node [style=none, fill=white, right, scale=0.7] (31) at (5.35, 3.0) {$Y$};
\node [] (32) at (3.75, 2.5) {};
\node [] (33) at (3.25, 2.25) {};
\node [] (34) at (4.25, 2.25) {};
\node [] (35) at (3.25, 0.0) {};
\node [style=none, fill=white, right, scale=0.7] (36) at (3.35, 2.0) {$X$};
\node [] (37) at (4.25, 0.75) {};
\node [style=none, fill=white, right, scale=0.7] (38) at (4.35, 2.0) {$X$};
\node [] (39) at (5.25, 1.5) {};
\node [] (40) at (5.25, 0.25) {};
\node [] (41) at (5.25, 0.0) {};
\node [style=none, fill=white, right, scale=0.7] (42) at (5.35, 0.0) {$Y$};
\node [circle, fill=black, scale=0.4] (43) at (5.25, 4.5) {};
\node [circle, fill=black, scale=0.4] (43) at (3.75, 2.5) {};
\node [circle, fill=black, scale=0.4] (43) at (5.25, 1.5) {};
\node [] (43) at (3.5, 3.25) {};
\node [] (44) at (5.5, 3.25) {};
\node [] (45) at (5.5, 3.75) {};
\node [] (46) at (3.5, 3.75) {};
\node [style=none, fill=white, scale=0.7] (47) at (4.25, 3.5) {$s$};
\node [] (48) at (4.0, 0.25) {};
\node [] (49) at (6.5, 0.25) {};
\node [] (50) at (6.5, 0.75) {};
\node [] (51) at (4.0, 0.75) {};
\node [style=none, fill=white, scale=0.7] (52) at (5.25, 0.5) {$t$};
\end{pgfonlayer}
\begin{pgfonlayer}{edgelayer}
\draw [in=90, out=-90] (1.center) to (2.center);
\draw [in=90, out=-90] (4.center) to (5.center);
\draw [in=90, out=-90] (7.center) to (8.center);
\draw [-, fill={white}] (10.center) to (11.center) to (12.center) to (13.center) to (10.center);
\draw [in=90, out=-90] (16.center) to (17.center);
\draw [in=90, out=180, looseness=0.5093] (19.center) to (20.center);
\draw [in=90, out=0, looseness=0.5093] (19.center) to (21.center);
\draw [in=90, out=-90] (17.center) to (19.center);
\draw [in=90, out=-90] (20.center) to (22.center);
\draw [in=90, out=-90] (21.center) to (24.center);
\draw [in=90, out=-90] (26.center) to (27.center);
\draw [in=90, out=-90] (29.center) to (30.center);
\draw [in=90, out=180, looseness=0.9391] (32.center) to (33.center);
\draw [in=90, out=0, looseness=0.9391] (32.center) to (34.center);
\draw [in=90, out=-90] (27.center) to (32.center);
\draw [in=90, out=-90] (33.center) to (35.center);
\draw [in=90, out=-90] (34.center) to (37.center);
\draw [in=90, out=-90] (30.center) to (39.center);
\draw [in=90, out=-90] (40.center) to (41.center);
\draw [-, fill={white}] (43.center) to (44.center) to (45.center) to (46.center) to (43.center);
\draw [-, fill={white}] (48.center) to (49.center) to (50.center) to (51.center) to (48.center);
\end{pgfonlayer}
\end{tikzpicture}
	\eeq
 \end{defi}   
 
 
 \begin{rem} 
   \emph{
   In the case of FinStoch, we see that a state $\pi_{X \times Y} \colon I \to X \times Y$ 
   corresponds to a probability distribution on $X \times Y$.
   Furthermore, a disintegration of $\pi_{X \times Y}$ is given by the 
   conditional distribution associated to the joint distribution as well as the 
   state $\pi_X \colon I \to X$ obtained by marginalizing $y$ in $\pi_{X \times Y}$.
   We define the conditional distribution
    $c \colon X \to Y$ explicitly by setting 
   \[c(x)(y) := \frac{\pi_{X \times Y}(x,y)}{\pi_X(x)}\]
   if $\pi_X(x)$ is not zero and in the event that $\pi_X(x) = 0$, 
   we set $c(x)(\_)$ to be any distribution on $Y$.} 
   
   \emph{Another category which admits conditionals is the category 
  BorelStoch which is a subcategory of the category Stoch whose objects 
  are standard Borel spaces.}
   \end{rem}
  
   Observe from the explicit calculation above that 
   the disintegration of a joint distribution is not necessarily 
   unique. 
   This leads us to the following definition.

  \begin{defi}
  \emph{Let $\pi_X \colon I \to X$ 
  be a state on an object $X \in \mathrm{Ob}(\mathcal{C})$. 
   Let $f,g \colon X \to Y$ be morphisms in $\mathcal{C}$. 
   We say that $f$ is almost surely equal to $g$ with 
   respect to $\pi_X$ or $f \sim_{\pi_X \mathrm{-a.s}} g$   
    if the following string diagrams coincide.}
    
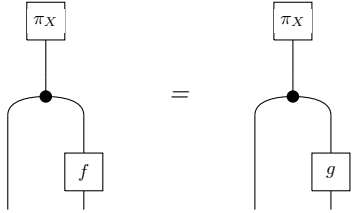
\[
		\begin{tikzpicture}[baseline=(0.base)]
\begin{pgfonlayer}{nodelayer}
\node (0) at (0, 1.5) {};
\node [] (1) at (0.5, 2.25) {};
\node [] (2) at (0.5, 1.75) {};
\node [] (3) at (0.5, 1.5) {};
\node [] (4) at (0.0, 1.25) {};
\node [] (5) at (1.0, 1.25) {};
\node [] (6) at (0.0, 0.0) {};
\node [] (7) at (1.0, 0.75) {};
\node [] (8) at (1.0, 0.25) {};
\node [] (9) at (1.0, 0.0) {};
\node [circle, fill=black, scale=0.485] (10) at (0.5, 1.5) {};
\node [] (10) at (0.25, 2.25) {};
\node [] (11) at (0.75, 2.25) {};
\node [] (12) at (0.75, 2.75) {};
\node [] (13) at (0.25, 2.75) {};
\node [style=none, fill=white, scale=0.7] (14) at (0.5, 2.5) {$\pi_{X}$};
\node [] (15) at (0.75, 0.25) {};
\node [] (16) at (1.25, 0.25) {};
\node [] (17) at (1.25, 0.75) {};
\node [] (18) at (0.75, 0.75) {};
\node [style=none, fill=white, scale=0.7] (19) at (1.0, 0.5) {$f$};
\node [style=none, fill=white] (20) at (2.25, 1.5) {=};
\node [] (21) at (3.75, 2.25) {};
\node [] (22) at (3.75, 1.75) {};
\node [] (23) at (3.75, 1.5) {};
\node [] (24) at (3.25, 1.25) {};
\node [] (25) at (4.25, 1.25) {};
\node [] (26) at (3.25, 0.0) {};
\node [] (27) at (4.25, 0.75) {};
\node [] (28) at (4.25, 0.25) {};
\node [] (29) at (4.25, 0.0) {};
\node [circle, fill=black, scale=0.485] (30) at (3.75, 1.5) {};
\node [] (30) at (3.5, 2.25) {};
\node [] (31) at (4.0, 2.25) {};
\node [] (32) at (4.0, 2.75) {};
\node [] (33) at (3.5, 2.75) {};
\node [style=none, fill=white, scale=0.7] (34) at (3.75, 2.5) {$\pi_{X}$};
\node [] (35) at (4.0, 0.25) {};
\node [] (36) at (4.5, 0.25) {};
\node [] (37) at (4.5, 0.75) {};
\node [] (38) at (4.0, 0.75) {};
\node [style=none, fill=white, scale=0.7] (39) at (4.25, 0.5) {$g$};
\end{pgfonlayer}
\begin{pgfonlayer}{edgelayer}
\draw [in=90, out=-90] (1.center) to (2.center);
\draw [in=90, out=180, looseness=0.9391] (3.center) to (4.center);
\draw [in=90, out=0, looseness=0.9391] (3.center) to (5.center);
\draw [in=90, out=-90] (2.center) to (3.center);
\draw [in=90, out=-90] (4.center) to (6.center);
\draw [in=90, out=-90] (5.center) to (7.center);
\draw [in=90, out=-90] (8.center) to (9.center);
\draw [-, fill={white}] (10.center) to (11.center) to (12.center) to (13.center) to (10.center);
\draw [-, fill={white}] (15.center) to (16.center) to (17.center) to (18.center) to (15.center);
\draw [in=90, out=-90] (21.center) to (22.center);
\draw [in=90, out=180, looseness=0.9391] (23.center) to (24.center);
\draw [in=90, out=0, looseness=0.9391] (23.center) to (25.center);
\draw [in=90, out=-90] (22.center) to (23.center);
\draw [in=90, out=-90] (24.center) to (26.center);
\draw [in=90, out=-90] (25.center) to (27.center);
\draw [in=90, out=-90] (28.center) to (29.center);
\draw [-, fill={white}] (30.center) to (31.center) to (32.center) to (33.center) to (30.center);
\draw [-, fill={white}] (35.center) to (36.center) to (37.center) to (38.center) to (35.center);
\end{pgfonlayer}
\end{tikzpicture}
	\]
    \end{defi}  
 
   Observe that if $\pi_{X \times Y} \colon I \to X \times Y$ is a state 
  and $\pi_X \colon I \to X$ is the associated marginal then 
  if $f,g \colon X \to Y$ are channels such that 
  $(f,\pi_X)$ and $(g,\pi_X)$ are both disintegrations 
  with respect to $\pi_{X \times Y}$ then 
  $f$ is $\pi_X$-almost surely equal to $g$.   

  \begin{defi}
  \emph{Let $\pi_X \colon I \to X$ 
  be a state on an object $X \in \mathrm{Ob}(\mathcal{C})$. 
  Let $f \colon X \to Y$ be a channel. The 
  Bayesian inversion of $f$ with respect to $\pi_X$ is a 
  channel $f^\dag_{\pi_X} \colon Y \to X$ such that we have the following equality of string diagrams.
  We say that $\mathcal{C}$ admits Bayesian inversions 
  if for every state $\pi_X \colon I \to X$ and
  $f \colon X \to Y$ we have a Bayesian inversion $f_{\pi_X}^\dag \colon Y \to X$.} 
  \end{defi} 
  
  Note that we can rephrase the definition above in 
  terms of disintegrations as follows. 
  Indeed, if $c$ and $\pi_X$ are as in the definition then 
  the Bayesian inversion $c^\dag_{\pi_X}$ can be obtained by 
  disintegrating the joint distribution $\pi_{X \times Y} \colon I \to Y \times X$ 
  obtained by swapping the integration of the pair $(c,\pi_X)$. 

  \begin{rem}
   \emph{As for disintegrations, we see that Bayesian inversions are not necessarily unique. 
   However, if $c_1,c_2$ are Bayesian inversions 
   of a channel $c \colon X \to Y$ 
   with respect to a state $\pi_X \colon I \to X$ 
   then $c_1$ is almost surely equal to $c_2$.}  
  \end{rem}
  
  
  \subsubsection{The category $\mathrm{PS}$} \label{section : PS}
  
  A crucial requirement of our set up that allows us to define 
  the BayesLearn functor similar to the Gradient learn functor 
  from \cite{CGgradientlearning} is that Bayesian inversions must 
  compose strictly. To this end, we must move away from 
  working with equivalence classes of almost surely equal morphisms 
  with respect to a given state and instead use the category
  $\mathrm{PS}$ (cf. \cite[Definition 13.8]{TFSyntheticApproach}).

  \begin{defi} 
 \emph{Suppose that 
  $\mathcal{C}$ is causal (cf.\cite[Definition 11.31]{TFSyntheticApproach}). Then 
  the category $\mathrm{ProbStoch}(\mathcal{C})$ 
  is defined as follows. 
  The objects of $\mathrm{ProbStoch}(\mathcal{C})$
  consist of pairs $(X,\pi_X)$
  where $X \in \mathrm{Ob}(\mathcal{C})$ and 
  $\pi_X \colon I \to X$ is a state on $X$. 
   A
    morphism in 
    $\mathrm{ProbStoch}(\mathcal{C})$
     between objects $(X,\pi_X)$ and $(Y,\pi_Y)$ consists of a map
     $f \colon X \to Y$ in $\mathcal{C}$ 
     satisfying $\pi_X ; f = \pi_Y$ modulo 
     $\pi_X$-a.s. equality, with 
     composition inherited from $\mathcal{C}$ i.e.
     \[ \mathrm{ProbStoch(\mathcal{C})}(X,Y) := 
    { \{f \in \mathcal{C}(X,Y) |  \pi_X ; f = \pi_Y \}} / { \sim_{\pi_X-a.s}} \]
     }
  \end{defi} 
 
  As mentioned in the introduction, we 
  write $\mathrm{PS}$ in place of $\mathrm{ProbStoch}$ for 
  ease of notation.

  \begin{rem}
    \emph{Note that the categories $\mathrm{Stoch}$ and $\mathrm{FinStoch}$ 
    are both causal. It is important to observe that if a category admits conditionals 
    then it is causal. However the converse is not true
    (cf. \cite[11.34, 11.35]{TFSyntheticApproach}).    
      }
  \end{rem} 
  
  \begin{rem} \label{rem : PS is SMC}
 \emph{Recall from \cite[Proposition 13.9(a)]{TFSyntheticApproach} that if
 $\mathcal{C}$ is causal then 
 $\mathrm{PS}(\mathcal{C})$ is 
 symmetric monoidal. In this case, the unit object is given 
 by the pair $(I,\iota)$ where 
 $\iota \colon I \to I$ is the identity map in $\mathcal{C}$.}
 \end{rem}     
       
\subsection{The Para construction} \label{section : the para construction}

 Recall that our goal is to understand conditional distributions 
 between random variables or morphisms
  in a Markov category that satisfies certain constraints.
 In this setting, we model a conditional distribution $p(y|x)$ using 
 a parametric function $f(x ;\theta)$ while our learning algorithm updates $\theta$
 using the given training set. The notion of parametrized function 
 has a natural formulation in category theory 
 which we call $\mathbf{Para}$ which was first introduced 
 in \cite{FSBackprop}. We use the more generalized version of this construction 
 which can be found in \cite{CGHR2021foundations}. 
 Observe that the type of the parameter $\theta$ need not coincide with that 
 of the variable. To ensure that this observation is 
 preserved in the categorical formulation, we 
 make use of the notion of actegories. 
 
 \begin{defi} \label{def : actegory}
  \emph{Let $(\mathcal{M},J,\star)$ be a symmetric monoidal category and 
  let $\mathcal{C}$ be a category.}
  \begin{enumerate}
  \item \emph{We say that $\mathcal{C}$ is an} $\mathcal{M}$-actegory \emph{if 
  we have a strong monoidal functor $\Phi \colon \mathcal{M} \to \mathrm{End}(\mathcal{C})$ where 
  $\mathrm{End}(\mathcal{C})$ is the 
  category of endofunctors on $\mathcal{C}$ for which the monoidal product is 
  given by composition.}
 \emph{Given $M \in \mathrm{Ob}(\mathcal{M})$ and $X \in \mathcal{C}$,
   we write $M \odot X := \Phi(M)(X)$.}
   \item \emph{We say that $\mathcal{C}$ is a} symmetric monoidal $\mathcal{M}$-actegory
   \emph{if in addition to 
   being an $\mathcal{M}$-actegory 
  $\mathcal{C}$ is endowed with natural isomorphisms}
  \[ \kappa_{M,X,Y} \colon M \odot ( X \otimes Y) \simeq X \otimes (M \odot Y) \]
  \emph{satisfying coherence laws reminiscent of the laws of a 
  costrong comonad.} 
   \end{enumerate}
 \end{defi}
 
 \begin{rem} \label{rem : mixed interchanger}
  \emph{ Part (2) of Definition \ref{def : actegory} is from 
   \cite[\S 2.1 Definition 4]{CGHR2021foundations}.
  As in this reference, we point out that if $(\mathcal{C},I,\otimes)$ is a symmetric 
  monoidal $(\mathcal{M},J,\star)$-actegory 
  then we have natural isomorphisms}
  \[\alpha_{M,X,Y} \colon M \odot (X \otimes Y) \simeq (M \odot X) \otimes Y \]
  \emph{and}
  \[\iota_{M,N,X,Y} \colon (M \star N) \odot (X \otimes Y) \simeq (M \odot X) \otimes (N \odot Y) \]
  \emph{which are called the 
  mixed associator and the mixed interchanger respectively.} 
 \end{rem}

  \begin{defi} \cite[\S Definition 2]{CGHR2021foundations}
 \emph{Let $\mathcal{M}$ be a symmetric monoidal category and let
  $\mathcal{C}$ be an $\mathcal{M}$-actegory. 
  The bicategory} $\mathbf{Para}_{\mathcal{M}}(\mathcal{C})$ 
  \emph{is defined as follows.}
  \begin{itemize}
  \item $\mathrm{Ob}(\mathbf{Para}_{\mathcal{M}}(\mathcal{C})) := \mathrm{Ob}(\mathcal{C})$.
  \item \emph{A 1-cell $f \colon X \to Y$ in $\mathbf{Para}_{\mathcal{M}}(\mathcal{C})$ consists 
  of a pair $(P,\phi)$ where $P \in \mathrm{Ob}(\mathcal{M})$ and 
  $\phi \colon P \odot X \to Y$ is a morphism in $\mathcal{C}$.   
   \item Let $(P,\phi) \in \mathbf{Para}_{\mathcal{M}}(\mathcal{C})(X,Y)$ and 
  $(Q,\psi) \in \mathbf{Para}_{\mathcal{M}}(\mathcal{C})(Y,Z)$. 
  The composition $(P,\phi) ; (Q,\psi)$ is the map in $\mathcal{C}$
  \[Q \odot ( P \odot X) \to Z \]
  given by 
  \[ Q \odot ( P \odot X) \xrightarrow{\phi} Q \odot Y \xrightarrow{\psi} Z \]}
  \item \emph{Let $(P,\phi), (Q,\psi) \in \mathbf{Para}_{\mathcal{M}}(\mathcal{C})(X,Y)$. 
  A 2-cell $\alpha \colon (P,\phi) \to (Q,\psi)$ is given by a morphism 
  $\alpha' \colon Q \to P$ such that the following diagram commutes.}
  $$
\begin{tikzcd} [row sep = large, column sep = large] 
Q \odot X \arrow[r, "\alpha' \odot \mathrm{id}_X"] \arrow[d, "\psi"] &
P \odot X \arrow[dl, "\phi"] \\
Y &  
\end{tikzcd}
$$

\item \emph{The identity and composition in the category 
$\mathbf{Para}_{\mathcal{M}}(\mathcal{C})(X,Y)$ are inherited from the identity and 
composition in $\mathcal{M}$.}
  \end{itemize} 
  \end{defi} 

By \cite[Proposition 3, \S 2]{CGHR2021foundations}, 
$\mathbf{Para}_{\mathcal{M}}(\_)$ defines a pseudo-monad on the 
category $\mathcal{M}-\mathbf{Mod}$ of $\mathcal{M}$-actegories. 
In particular, if we have a functor 
$$F \colon \mathcal{C} \to \mathcal{D}$$
 we get 
an associated 
functor 
\[ \mathbf{Para}_{\mathcal{M}}(F) \colon \mathbf{Para}_{\mathcal{M}}(\mathcal{C}) \to \mathbf{Para}_{\mathcal{M}}(\mathcal{D}) \]
We make use of this fact when we define the BayesLearn functor in \S \ref{section : bayes learn}.

\section{Bayes Learn} \label{section : bayes learn}

  In this section we outline the construction 
  of the functor BayesLearn which aims to capture the essential features of Bayesian learning.
  However in order for us to discuss these results, we 
  require certain preliminary ideas which allow
  us to better understand the actegory structure 
  on $\mathrm{PS}(\mathcal{C})$ where 
  $\mathcal{C}$ is a Markov category. 
  \subsection{ Inducing the $\mathcal{M}$-actegory structure}

   While we work with the flexibility provided by $\mathbf{Para}_{\mathcal{M}}$, we 
   must be careful to ensure that the categories to which we apply 
   $\mathbf{Para}_{\mathcal{M}}(\_)$
   are endowed with the structure of an $\mathcal{M}$-actegory. 
    To this end, we must make certain modifications to the category $\mathcal{M}$ itself as we 
    would 
    firstly like to ensure that parameter spaces have a well defined prior associated to them.

 \begin{rem}     
 \emph{ It follows directly from the definition of an actegory that if 
 $r \colon M \to N$ in $\mathcal{M}$ and $f \colon X \to Y$ in
      $\mathcal{C}$
      then
       the following diagram commutes.}
      $$
\begin{tikzcd} [row sep = large, column sep = large] 
 M \odot X \arrow[r, "M \odot f"] \arrow[d, "r \odot X"] &
M \odot Y \arrow[d, "r \odot Y"] \\
N \odot X \arrow[r, "N \odot f"] & N \odot Y  
\end{tikzcd}
$$ 
\emph{where the vertical morphisms are 
due to the  natural transformation between the functors 
$M \odot \_ $ and $N \odot \_$ induced by 
the morphism $M \to N$. We write 
$f \odot g$ to denote the composition
$M \odot X \to M \odot Y  \to N \odot Y$ or equivalently 
$M \odot X \to N \odot X  \to N \odot Y$.} 
\end{rem}
      
    The following technical condition is required in the proof of 
      Lemma \ref{lem : PS(M) actegories}. 
      
   \begin{defi} \label{def : agreement}
   \emph{ Let ($\mathcal{M},J,\star)$
    and $(\mathcal{C},I,\otimes)$ be Markov categories and 
   let $\mathcal{C}$ be a symmetric monoidal $\mathcal{M}$-actegory (cf. Definition \ref{def : actegory}). 
   We say that $\mathcal{C}$
   is in} agreement \emph{with $\mathcal{M}$ if}
    \emph{for every $P \in \mathcal{M}$ and 
     $X \in \mathcal{C}$, 
    the following diagram is commutative and natural 
    in both $P$ and $X$.}
$$
\begin{tikzcd} [row sep = 8ex, column sep = 10em] 
 P \odot X \arrow[r, "\mathrm{copy}_P \odot \mathrm{copy}_X"] \arrow[d,equal] &
(P \star P) \odot (X \otimes X) \arrow[d, "\iota_{P,P,X,X}"] \\
P \odot X \arrow[r, "\mathrm{copy}_{P \odot X}"] & (P \odot X) \otimes (P \odot X) 
\end{tikzcd}
$$ 
\emph{where $\iota_{P,P,X,X}$ is the mixed interchanger as 
introduced in \ref{rem : mixed interchanger}.}
   \end{defi}
   
    
       
       Recall from \ref{section : bayesian inference prelims},
       that a categorical formulation of Bayesian learning must 
      capture the notion of state on the parameter as well as a means
      to update it via Bayesian inversion. 
      It is hence natural under these circumstances to 
      utilize $\mathrm{PS}(\mathcal{M})$
      as the category of parameters. 
      A brief discussion of the PS construction 
      was made in \ref{section : PS}.

   \begin{lem} \label{lem : PS(M) actegories}
     Let $(\mathcal{M},J,\star)$ be a causal Markov category 
     and $(\mathcal{C},I,\otimes)$ be 
     a Markov category that is in agreement with $\mathcal{M}$.
      The category $\mathrm{PS}(\mathcal{C})$ is a
      $\mathrm{PS}(\mathcal{M})$-actegory.
         \end{lem} 
   \begin{proof} 
   Let $(P,\pi_P) \in \mathrm{Ob}(\mathrm{PS}(\mathcal{C}))$. We
   define an endofunctor 
   \[(P,\pi_P) \odot \_ \colon \mathrm{PS}(\mathcal{C}) \to \mathrm{PS}(\mathcal{C}) \] as follows.
   Firstly, let $(X,\pi_X)$ be an object in $\mathrm{PS}(\mathcal{C})$. 
   We set $(P,\pi_P) \odot (X,\pi_X)$ to be the pair 
   \[ I \xrightarrow{\sim} J \odot I \xrightarrow{\pi_P \odot \pi_X} P \odot X\] 
   where the first isomorphism is due to the $\mathcal{M}$-actegory
   structure on $\mathcal{C}$. 
   Observe that the composition 
   above coincides with 
   \begin{align} \label{obs1}
    I \xrightarrow{\sim} J \odot I \xrightarrow{\pi_P \odot \mathrm{id}} P \odot I \xrightarrow{ \mathrm{id} \odot \pi_X} P \odot X.
    \end{align}
   We check that $(P,\pi_P) \odot \_$ defines a functor.
       Suppose
   $f \colon (X,\pi_X) \to (Y,\pi_Y)$ in $\mathrm{PS}(\mathcal{C})$.
   The morphism $\pi_Y \colon I \to Y$ coincides with the composition 
   $I \xrightarrow{\pi_X} X \xrightarrow{\tilde{f}} Y$
   where $\tilde{f}$ is a representative of the equivalence class of 
   $f$ in $\mathcal{C}$.
   Applying $P \odot \_$ gives 
   \[ P \odot I \xrightarrow{\mathrm{id} \odot \pi_X} P \odot X \xrightarrow{\mathrm{id} \otimes \tilde{f} } P \odot Y. \]
  Hence we get a sequence of morphisms
  \begin{align} \label{composition1}
  I \xrightarrow{\sim} J \odot I \xrightarrow{\pi_P \odot \mathrm{id}}  P \odot I \xrightarrow{\mathrm{id} \odot \pi_X} P \odot X \xrightarrow{\mathrm{id} \odot \tilde{f}} P \odot Y. 
  \end{align}
  Using the observation in Equation (\ref{obs1})
  and the fact that 
  $\pi_X ; \tilde{f} = \pi_Y$, we deduce that 
  Equation (\ref{composition1}) coincides with 
  $(P,\pi_P) \odot (Y,\pi_Y)$. 
  We have thus shown that we have a 
  map 
  \[(P,\pi_P) \odot \tilde{f} \colon (P,\pi_P) \odot (X,\pi_X) \to (P,\pi_P) \odot (Y,\pi_Y) \]
   in $\mathrm{PS}(\mathcal{C})$. 
 Note that we must show that the map 
  $(P,\pi_P) \odot (X,\pi_X) \to (P,\pi_P) \odot (Y,\pi_Y)$ we have defined above is independent of our choice of $\tilde{f}$. 
 Hence we verify 
  that if 
  $h,g \colon X \to Y$ are morphisms in $\mathcal{C}$ such that 
  $h \sim_{\pi_X \mathrm{-a.s}} g$ then 
  $(P,\pi_P) \odot h \sim_{\pi_P \odot\pi_X \mathrm{-a.s}} (P,\pi_P) \odot g$. 
  Since $h \sim_{\pi_X\mathrm{-a.s}} g$, we get that 
  the compositions 
  \[ A  := I \to X \to X \otimes X \xrightarrow{h \otimes \mathrm{id}} Y \otimes X\]
  and 
   \[B := I \to X \to X \otimes X \xrightarrow{g \otimes \mathrm{id}} Y \otimes X\]
   coincide.
   Consider the sequence 
   \[C := J \xrightarrow{\pi_P} P \xrightarrow{\mathrm{copy}_P} P \otimes P \xrightarrow{\mathrm{id}} P \otimes P \]
   in $\mathcal{M}$. 
   Since $A = B$, we see that 
   $C \odot A = C \odot B$. 
   However, this implies that 
   \[I \to P \odot X \xrightarrow{\mathrm{copy}_{P} \odot \mathrm{copy}_{X}} (P \star P) \odot (X \otimes X) 
   \xrightarrow{\mathrm{id} \odot (h \otimes \mathrm{id})}
    (P \star P) \odot (Y \otimes X) \]
   coincides with 
   \[I \to P \odot X \xrightarrow{ \mathrm{copy}_{P} \odot \mathrm{copy}_{X}} (P \star P) \odot (X \otimes X) \xrightarrow {\mathrm{id} \odot (g \otimes \mathrm{id})} (P \star P) \odot (Y \otimes X). \]
   It follows that 
    \[I \to P \odot X \xrightarrow{\mathrm{copy}_{P \odot X}} (P \odot X) \otimes (P \odot X) 
   \xrightarrow{(P \odot h) \otimes \mathrm{id}}
    (P \odot Y) \otimes (P \odot X) \]
    coincides with 
    \[I \to P \odot X \xrightarrow{\mathrm{copy}_{P \odot X}} (P \odot X) 
    \otimes (P \odot X) \xrightarrow {(P \odot g) \otimes \mathrm{id}} (P \odot Y) \otimes (P \odot X). \]
    This is a consequence of the naturality of the 
    agreement (cf. Definition \ref{def : agreement})
    and mixed interchanger morphisms 
    (cf. Remark \ref{rem : mixed interchanger}). 
   This verifies 
   that 
   $(P,\pi_P) \odot f \sim_{\pi_P \odot \pi_X \mathrm{-a.s}} (P,\pi_P) \odot g$.
  One checks without difficulty that 
  $(P,\pi_P) \odot \_$ respects composition of morphisms, identities
  and that
  \[(P,\pi_P) \odot \_ \colon \mathrm{PS}(\mathcal{C}) \to \mathrm{PS}(\mathcal{C}). \]
   is indeed a well defined functor. 
  Furthermore,
  using the notation from Remark \ref{rem : PS is SMC},
   $(J,\iota) \odot \_$ coincides with the identity 
   on $\mathrm{PS}(\mathcal{C})$ and 
   we have a natural isomorphism 
   $\mathrm{id}_{\mathrm{PS}(\mathcal{C})} \xrightarrow{\sim}  (J,\iota) \odot \_$
   which is induced by the 
   isomorphism 
   $\mathrm{id}_{\mathcal{C}} \xrightarrow{\sim}  J \odot \_.$
   
   Lastly,
   let $(P,\pi_P), (Q,\pi_Q) \in \mathrm{PS}(\mathcal{M})$
   and $(X,\pi_X) \in \mathrm{PS}(\mathcal{C})$.
   Since
   we have a strong monoidal functor 
   $P \odot ( Q \odot \_ ) \xrightarrow{\sim} (P \star Q) \odot \_$, 
   we deduce that the following diagram is commutative. 
$$
\begin{tikzcd} [row sep = large, column sep = 12ex] 
I \arrow[r,"\sim"] \arrow[d,equal] & J \odot I \arrow[r] \arrow[d,equal] & J \odot (J \odot I) \isoarrow{d} \arrow[r, "\pi_P \odot (\pi_Q \odot \pi_X)"] & P \odot (Q \odot X ) \isoarrow{d} \\
I \arrow[r,"\sim"]  & J \odot I \arrow[r] & ( J \star J ) \odot I \arrow[r,"(\pi_P \star \pi_Q) \odot \pi_X"] & 
( P \star Q) \odot X.  
\end{tikzcd}
$$
We deduce from this and similar such arguments that 
we have an isomorphism 
\[ (P,\pi_P) \odot ((Q,\pi_Q) \odot \_ ) \xrightarrow{\sim} ((P,\pi_P) \star (Q,\pi_Q)) \odot \_ \]
making the 
functor
\[\mathrm{PS}(\mathcal{M}) \to \mathrm{End}(\mathrm{PS}(\mathcal{C}))\]
given by 
\[ (P,\pi_P) \mapsto (P,\pi_P) \odot \_ \] 
a strong monoidal functor. 
This concludes the proof. 
       \end{proof}

\begin{lem} \label{lem : PS(M) actegories II}
Let $(\mathcal{A},I,\otimes)$ and $(\mathcal{M},J,\star)$ be symmetric monoidal 
categories such that $\mathcal{A}$ is an $\mathcal{M}$-actegory. 
Let $$S \colon \mathcal{A}^{\mathrm{op}} \to \mathcal{M}-\mathrm{Mod}$$ be a functor which takes values in the category of $\mathcal{M}$-actegories.
We suppose that $S$ satisfies the following property. 
\begin{enumerate}
\item There exists an $\mathcal{M}$-actegory $\mathcal{B}$ such that for every 
$X \in \mathrm{Ob}(\mathcal{A})$, 
$S(X) = \mathcal{B}$.
\item For every $M \in \mathrm{Ob}(\mathcal{M})$, 
$X,Y \in \mathrm{Ob}(\mathcal{A})$, $f \in \mathcal{A}(X,Y)$,
$A \in \mathrm{Ob}(S(X))$ and $B \in \mathrm{Ob}(S(Y))$, 
we have that 
\[ M \odot S(f)(B) = S(M \odot f) (M \odot B) \]
\end{enumerate}
The Grothendieck lens $\mathrm{Lens}_S$ is then an $\mathcal{M}$-actegory. 
     \end{lem}
Note that in the statement of the above lemma, condition 
(2) only makes sense if we have condition (1). 
\begin{proof}
 Recall that the objects of $\mathrm{Lens}_S$ are tuples of the form 
 $(X, A)$ where $X$ is an object in $\mathcal{A}$ and 
 $A \in \mathrm{Ob}(S(X))$. Given $M \in \mathrm{Ob}(\mathcal{M})$, 
 we define 
  \[ M \odot (X,A) := (M \odot X, M \odot A).\]
 We show that 
 $M \odot \_$ defines an endofunctor 
on $\mathrm{Lens}_S$. 
Let $\mathbf{f} \colon (X,A) \to (Y,B)$ be given 
by a pair of morphisms 
$f \colon X \to Y$ and a map $f^* \colon S(f)(B) \to A$. 
We define 
\[M \odot \mathbf{f} \colon M \odot (X,A) \to M \odot (Y,B)\]
in $\mathrm{Lens}_S$
as follows. 
Since $\mathcal{M}$ is a $\mathcal{M}$-actegory, we have a morphism 
\[M \odot {f} \colon M \odot X \to M \odot Y\]
and a morphism 
\[M \odot {f^*} \colon M \odot S(f)(B) \to M \odot A.\]
We apply the assumption on $S$ to get
\[M \odot {f^*} \colon S(M \odot f)(M \odot B) \to M \odot A.\]
The pair 
$(M \odot {f}, M \odot {f^*})$ defines a morphism 
$M \odot (X,A) \to M \odot (Y,B)$ in $\mathrm{Lens}_S$. 
Clearly, $M \odot \_$ preserves identities. One checks without difficulty 
that $M \odot \_$ respects compositions as well. We have thus shown 
that $M \odot \_$ defines an endofunctor on $\mathrm{Lens}_S$. 

It remains to show that the morphism
$\mathcal{M} \to \mathrm{End}(\mathrm{Lens}_S)$ is a strong 
monoidal functor. Firstly, 
using that $\mathcal{A}$ is a $\mathcal{M}$-actegory 
and for every $X \in \mathrm{Ob}(\mathcal{A})$,  
$S(X)$ is a $\mathcal{M}$-actegory, we deduce that 
if $J$ is the unit object of $\mathcal{M}$ then 
$J \odot \_$ is naturally isomorphic to the identity 
endofunctor. 
In a similar fashion, 
given $N,M \in \mathrm{Ob}(\mathcal{M})$ and 
$(X,A)$ in $\mathrm{Lens}_S$
we verify 
that 
we have a natural isomorphism of 
endofunctors 
\[N \odot (M \odot (X,A)) \xrightarrow{\sim} (N \star M) \odot (X,A).\]
This concludes the proof. 
\end{proof} 
   
\subsection{Bayes Learn}

  Let $\mathcal{C}$ be a Markov category which admits 
  conditionals. Note that this is equivalent to saying that $\mathcal{C}$ admits 
  Bayesian inversions. 
  Since Bayesian inversions in general are defined 
  up to an equivalence relation, we restrict our attention to the 
  category $\mathrm{PS}(\mathcal{C})$ (cf.\S \ref{section : PS}). 
  In fact, in this case, Bayesian inversion defined a symmetric monoidal dagger functor 
  on $\mathrm{PS}(\mathcal{C})$ in the sense of dagger categories. 
  We refer the reader to \cite[Remark 13.10]{TFSyntheticApproach}.
  
Recall from \cite{CGgradientlearning}, the 
basis of the gradient learning 
functor $\mathrm{GL}$ comes from a functor 
\[R \colon \mathcal{C} \to \mathrm{Lens}(\mathcal{C})\]
where in this case $\mathcal{C}$ is a
Cartesian reverse differential category.
In our situation, we can mirror this construction
via the mechanism of Bayesian inversion and the notion 
of generalized lenses. 
We proceed below in greater detail.

 As above, our goal is to define 
  a functor
  \[ R \colon \mathrm{PS}(\mathcal{C}) \to \mathrm{Lens}_F \]
  where $\mathrm{Lens}_F$ is the $F$-lens associated 
  to a functor
  \[\mathrm{PS}(\mathcal{C})^{\mathrm{op}} \to \mathrm{Cat} \]
  (cf. \cite{S2020generalized}). 

\subsubsection{The functor $\mathrm{Stat}$}
  We define the functor 
 \[\mathrm{Stat} \colon \mathrm{PS}(\mathcal{C})^{\mathrm{op}} \to \mathrm{Cat}\]
 as follows.
  Given $X \in \mathrm{Ob}(\mathrm{PS}(\mathcal{C}))$,
   let
   \[\mathrm{Stat}(X) := \mathrm{PS}(\mathcal{C}).\] 
   Given a map $f \colon X \to Y$ in $\mathrm{PS}(\mathcal{C})$,
   the natural transformation 
   $F(Y) \to F(X)$ is the identity functor.

   \begin{rem} \label{rem : morphisms simplify}
   \emph{Observe that our definition of $\mathrm{Stat}$ coincides with the definition 
   of $\mathrm{Stat}$ from 
   \cite{Sm2020bayesian}. Indeed, 
   since $\mathcal{C}$ is Markov, the unit $I$ is a terminal object and 
 hence there is a unique state $\iota \colon I \to I$. We simplify notation as above 
 and write $I$ in place of $(I,\iota)$. 
 $\mathrm{PS}(\mathcal{C})(I,(X,\pi_X))$ hence consists of a single element 
  corresponding to the state $\pi_X$. 
If $X \in \mathrm{PS}(\mathcal{C})$ 
   and $(A,\pi_A),(B,\pi_B) \in \mathrm{Ob}(\mathrm{Stat}(X))$ then 
   a morphism $f \colon (A,\pi_A) \to (B,\pi_B)$ in $\mathrm{Stat}(X)$
  coincides with a morphism of sets 
  \[ \mathrm{PS}(\mathcal{C})(I,(X,\pi_X)) \to \mathrm{PS}(\mathcal{C})((A,\pi_A),(B,\pi_B)) \]
   since $\mathrm{PS}(\mathcal{C})(I,(X,\pi_X))$ is the singleton set.
   This is precisely how 
   $\mathrm{Stat}$ is defined in \cite{Sm2020bayesian}.}
      \end{rem}
  
  \subsubsection{The Grothendieck Lens}
  Let $\mathrm{Lens}_\mathrm{Stat}$ be the lens associated to the functor 
  $\mathrm{Stat}$ as introduced in \cite[\S 3.1]{S2020generalized}.
   More precisely, we have that 
   \begin{itemize}
   \item The objects of the category $\mathrm{Lens}_{\mathrm{Stat}}$ are 
   pairs $((X,\pi_X),(A,\pi_A))$ where 
   $(X,\pi_X) \in \mathrm{PS}(\mathcal{C})$ and 
   $(A,\pi_A) \in \mathrm{Stat}(X)$. 
   \item A morphism $\phi \colon ((X,\pi_X),(A,\pi_A)) \to ((Y,\pi_Y),(B,\pi_B))$ 
   is given by a morphism $(X,\pi_X) \to (Y,\pi_Y)$ in $\mathrm{PS}(\mathcal{C})$
   and a morphism $(B,\pi_B) \to (A,\pi_A)$ in \\ $\mathrm{Stat}(X) = \mathrm{PS}(\mathcal{C})$.    
 \end{itemize}    
   One checks that this is a well defined category and 
    is an
   instance of the \emph{Grothendieck construction} as in \cite[\S 3.1]{S2020generalized}.
   \begin{rem} \label{rem : Lens_Stat is simple}
   \emph{The category $\mathrm{Lens}_{\mathrm{Stat}}$ simplifies 
   considerably in our situation.}
   \[ \mathrm{Lens}_{\mathrm{Stat}} \simeq 
   \mathrm{PS}(\mathcal{C}) \times \mathrm{PS}(\mathcal{C})^{\mathrm{op}} \]
   \end{rem}

 \subsubsection{ The functor \emph{R}}
  We define the functor 
  \[R \colon \mathrm{PS}(\mathcal{C}) \to \mathrm{Lens}_{\mathrm{Stat}}\]
  as follows. 
  \begin{itemize}
  \item  Given $(X,\pi_X) \in \mathrm{PS}(\mathcal{C})$, we set 
  $R((X,\pi_X)) := ((X,\pi_X),(X,\pi_X))$.
  \item If $f \colon (X,\pi_X) \to (Y,\pi_Y)$ is a morphism 
  in $\mathrm{PS}(\mathcal{C})$ then 
  the map 
  \[R(f) \colon ((X,\pi_X),(X,\pi_X)) \to ((Y,\pi_Y),(Y,\pi_Y)) \]
  in $\mathrm{Lens}_{\mathrm{Stat}}$
   is defined 
   to be the pair 
   $(f,f_{\pi_X}^{\dag})$ where 
   $f^{\dag}_{\pi_X}$ is the Bayesian inversion 
   of $f$ with respect to the state $\pi_X$ on $X$.
    Note that this simplification is a consequence of our discussion above.
    \end{itemize} 
%
 
 \begin{prop}
  The functor $R$ is well defined. 
 \end{prop} 
  \begin{proof}
   We must essentially verify that $R$ behaves well with regards to composition. 
  This is a consequence of the fact that in the category $\mathrm{PS}(\mathcal{C})$,
  Bayesian inversions are unique and 
  their composition is indeed well defined.   
  \end{proof} 
  
  \subsubsection{The functor \emph{BayesLearn}} \label{section : BayesLearn}
   
   Let $\mathcal{M}$ and $\mathcal{C}$ be Markov categories with $\mathcal{M}$ causal. 
  To define {BayesLearn}, we now specialize to the 
  case where $\mathcal{C}$ is a symmetric monoidal $\mathcal{M}$-actegory which in addition 
  is in agreement with $\mathcal{M}$. 
 
 Since $\mathcal{M}$ is causal, $\mathrm{PS}(\mathcal{M})$ 
 is a well defined symmetric monoidal category.
  By Lemmas \ref{lem : PS(M) actegories} 
  and
  \ref{lem : PS(M) actegories II}, the 
  categories 
  $\mathrm{PS}(\mathcal{C})$ and 
  $\mathrm{Lens}_{\mathrm{Stat}}$ are 
  $\mathrm{PS}(\mathcal{M})$-actegories. 
  Recall from
   Section \ref{section : the para construction} that
    $\mathbf{Para}_{\mathrm{PS}(\mathcal{M})}(\_)$ is a well defined functor 
    which when applied to $R$ gives a functor 
  \[ \mathbf{Para}_{\mathrm{PS}(\mathcal{M})}(R) 
  \colon \mathbf{Para}_{\mathrm{PS}(\mathcal{M})}(\mathrm{PS}(\mathcal{C})) \to \mathbf{Para}_{\mathrm{PS}(\mathcal{M})}(\mathrm{Lens}_{\mathrm{Stat}}) \]  
  
  Recall that 
  if $(\mathcal{P},J,\star)$ is a symmetric monoidal category
  then
  a $\mathcal{P}$-actegory $\mathcal{A}$
  admits a canonical functor 
  $j_{\mathcal{P},\mathcal{A}} \colon \mathcal{A} \to \mathbf{Para}_{\mathcal{P}}(\mathcal{A})$
  given by $A \mapsto J \odot A$. 
  Note that  $j_{\mathcal{P},\mathcal{A}}$
  is the unit for the pseudo-monad 
  defined by $\mathbf{Para}_{\mathcal{P}}(\cdot)$. 
  We thus have a diagram
$$
\begin{tikzcd} [row sep = large, column sep = 12ex] 
\mathcal{C} \arrow[r] & \mathrm{PS}(\mathcal{C}) \arrow[r,"R"] \arrow[d,"j_{\mathrm{PS}(\mathcal{M}),\mathcal{C}}"] &
\mathrm{Lens}_{\mathrm{Stat}} \arrow[d,"j_{\mathrm{PS}(\mathcal{M}),\mathrm{Lens}_{\mathrm{Stat}}}"] \\
& \mathbf{Para}(\mathrm{PS}(\mathcal{C})) \arrow[r,"\mathbf{Para}_{\mathrm{PS}(\mathcal{M})}(R)"] & \mathbf{Para}(\mathrm{Lens}_{\mathrm{Stat}})
\end{tikzcd}
$$ 
 
  \begin{defi} 
  \emph{We define}
   \[\mathrm{BayesLearn} := \mathbf{Para}_{\mathrm{PS}(\mathcal{M})}(R).\]
   \end{defi}
 
  \begin{rem}
 \emph{Observe that the BayesLearn functor in 
  \S \ref{section : BayesLearn}
   does not have an update or displacement 
  endofunctor as the Gradient learning functor 
  from \cite[\S 3.5]{CGgradientlearning}. 
  This is due to the relatively simplified nature of Bayesian learning where
  parameter updates correspond to obtaining the posterior distribution
  using the prior and the likelihood and not
   as a result of optimizing with 
  respect to a loss function. 
  Equivalently, in the categorical setting, Bayesian inversion 
  provides an update rule without recourse to a displacement or error endofunctor.}
  \end{rem}
  
 \begin{rem} \label{rem : bayesian learning easiest}
  \emph{Note that when we contrast the Bayesian learning framework 
  described above with the gradient learning framework 
  described in \cite{CGgradientlearning}, we 
  see that Bayesian learning is considerably simpler.
  By using the Grothendieck lens in place of the standard Lens construction
  and $\mathrm{ProbStoch}(\mathcal{C})$ in place 
  of $\mathcal{C}$, the existence of Bayesian inversion 
  as a dagger functor 
  implies the breakdown of the functor Stat resulting 
  in the simplified 
  form of $\mathrm{Lens}_{\mathrm{Stat}}$ 
  as described in 
  Remark \ref{rem : Lens_Stat is simple}. In addition, the absence 
  of error and update endofunctors further distinguishes 
  the Bayes Learning framework. In this sense, we believe Bayesian learning 
  to be the simplest case of the categorical learning as described 
  in \cite{CGgradientlearning}.
  }
 \end{rem} 
  
   \begin{rem}
 \emph{Lemma \ref{lem : PS(M) actegories II} 
 provides $\mathrm{PS}(\mathcal{C})$ with the structure 
 of a $\mathrm{PS}(\mathcal{M})$-actegory. 
 However, by defining Stat for the 
 category $\mathcal{M}$ which we denote $\mathrm{Stat}_\mathcal{M}$, we also get that}
 \[ \mathrm{Lens}_{\mathrm{Stat}_\mathcal{M}} 
     \simeq \mathrm{PS}(\mathcal{M}) \times \mathrm{PS}(\mathcal{M})^{\mathrm{op}}. \]
\emph{Hence
 $\mathrm{Lens}_{\mathrm{Stat}_\mathcal{\mathcal{C}}}$ 
 in fact 
 has the 
 structure of $\mathrm{Lens}_{\mathrm{Stat}_\mathcal{M}}$-actegory.
 Since we are interested only in the BayesLearn functor and in particular its image, we
 do not concern ourselves with this more general action of 
 $\mathrm{Lens}_{\mathrm{Stat}_\mathcal{M}}$.}
 
 \emph{Note also that we can endow $\mathrm{PS}(\mathcal{C})$
 with the structure of
 $\mathrm{Lens}_{\mathrm{Stat}_{\mathrm{PS}(\mathcal{M})}}$-actegory 
 via the projection 
 \[\mathrm{Lens}_{\mathrm{Stat}_{\mathcal{M}}}
 \to \mathrm{PS}(\mathcal{M}).\] This 
 implies that 
 we can also view $\mathrm{Lens}_{\mathrm{Stat}_{\mathcal{C}}}$
 as a
 $\mathrm{Lens}_{\mathrm{Stat}_{\mathcal{M}}}$-actegory in line with the theory from 
 \cite{CGHR2021foundations}.}
 \end{rem}

%
%

\subsection{Bayes Learning algorithm} \label{sec : bayes learning algo}

We now detail the Bayes Learning algorithm in our current setup. 
We preserve our assumptions on $\mathcal{C}$ and 
$\mathcal{M}$ from the previous section.
We are given training data which consists 
of objects $X_T$ and $Y_T$
in $\mathcal{C}$ and a joint distribution $\omega_T \colon I \to X_T \otimes Y_T$.  
We would like to perform inference 
for objects $X_*$ and $Y_*$ in $\mathcal{C}$.
We are also provided with states 
$\pi_{X_*} \colon I \to X_*$
and $\pi_{X_T} \colon I \to X_T$ where $\pi_{X_T}$ is
obtained by marginalizing $\omega_T$. 
Note that in practice 
$X_T = X_*$ and $Y_T = Y_*$
We proceed as follows. 

\begin{enumerate} 
\item We model the given data 
by choosing a function 
$f \colon X_T \to Y_T$ in 
$\mathbf{Para}_{\mathcal{M}}(\mathcal{C})$.
This corresponds to a 
morphism 
$f \colon M \odot X_T \to Y_T$. We assume that 
model satisfies a technical assumption which we precise in (2).
We assume we have a similar model applicable
for the inference data i.e. $f_* \colon M \odot X_* \to Y_*$. 
\item The morphism 
$f$ descends to give a morphism
in the category $\mathrm{PS}(\mathcal{C})$ as follows. 
Firstly, we endow $M$ with a prior distribution i.e.
a state $\pi_M \colon J \to M$.
The morphism $\pi_M \odot \pi_{X_T} \colon I \simeq J \odot I \to M \odot X_T$ defines 
a state on $M \odot X_T$.
By composing with $f$ we obtain a state on $Y_T$ 
i.e. a morphism $\pi_{Y_T} := (\pi_M \odot \pi_{X_T}) ; f \colon I \to Y_T$.
We assume that the model $f$ and the prior $\pi_M$ were 
chosen so as to guarantee that $\pi_{Y_T}$ coincides with the marginal distribution
on $Y_T$ from $\omega_T$.
The equivalence class of $f$ defines a morphism 
$(M,\pi_M) \odot (X_T,\pi_{X_T}) \to (Y_T,\pi_{Y_T})$.
Equivalently, we have a map in 
$\mathbf{Para}_{\mathrm{PS}(\mathcal{M})}(\mathrm{PS}(\mathcal{C}))$.
\item By construction, 
\[\mathrm{BayesLearn}(f) := (f,f^\dag) \]
is a morphism 
in $\mathbf{Para}_{\mathrm{PS}(\mathcal{M})}(\mathrm{Lens}_{\mathrm{Stat}})$ between 
objects $((X_T,\pi_{X_T}),(X_T,\pi_{X_T}))$ and $((Y_T,\pi_{Y_T}),(Y_T,\pi_{Y_T}))$. 
Let \[f^\dag \colon (Y_T,\pi_{Y_T}) \to (M,\pi_{M}) \odot (X_T,\pi_{X_T})\] 
denote the corresponding inversion. 
\item Recall that in \S \ref{section : bayesian inference prelims}, 
  we outlined how to leverage the posterior distribution to 
  make predictions. 
  This can be formalized in 
  a categorical setting as follows. Consider the composition
\begin{align*} \label{equn : predictive density}
 (Y_T,\pi_{Y_T}) \otimes (X_*,\pi_{X_*}) & \xrightarrow{f^\dag \otimes \mathrm{id}}
    ((M,\pi_M) \odot (X_T,\pi_{X_T})) \otimes (X_*,\pi_{X_*}) 
    \\ & \simeq (X_T,\pi_{X_T}) \otimes ( (M,\pi_M) \odot (X_*,\pi_{X_*}))
\end{align*}  
The isomorphism above is obtained by composing the isomorphisms 
\begin{align*}
((M,\pi_M) & \odot (X_T,\pi_{X_T})) \otimes (X_*,\pi_{X_*})
\\ & \overset{(i)} \simeq (((M,\pi_M) \odot I) \otimes (X_T,\pi_{X_T})) \otimes (X_*,\pi_{X_*})
\\ & \overset{(ii)} \simeq ((X_T,\pi_{X_T}) \otimes ((M,\pi_M) \odot I)) \otimes (X_*,\pi_{X_*})  
\\ & \overset{(iii)} \simeq (X_T,\pi_{X_T}) \otimes (((M,\pi_M) \odot I) \otimes (X_*,\pi_{X_*}))
\\ & \overset{(iv)} \simeq (X_T,\pi_{X_T}) \otimes ((M,\pi_M) \odot (X_*,\pi_{X_*}))
\end{align*}
where 
(i) is due to the mixed associator, (ii) is because of the swap
isomorphism, (iii) is a consequence of the associative property 
of the monoidal product and (iv) is obtained by 
applying the mixed associator again. 

   We conditionalize the morphsim to get a morphism 
\[(X_T,\pi_{X_T}) \otimes  (Y_T,\pi_{Y_T}) \otimes (X_*,\pi_{X_*})
 \to (M,\pi_M) \odot (X_*,\pi_{X_*}) \]
 By composing on the right by $f_*$, we get 
 \[(X_T,\pi_{X_T}) \otimes  (Y_T,\pi_{Y_T}) \otimes (X_*,\pi_{X_*}) \to (Y_*,\pi_{Y_*}).\]

  This is the Bayes predictive distribution. 
  By pre-composing with the state 
  $\pi_{X_T} \otimes \pi_{Y_T}$, we 
  get a map $(X_*,\pi_{X_*}) \to (Y_*,\pi_{Y_*})$. We call this the full predictive distribution
  obtained by averaging out the predictive distributions as they vary over 
  different instances of the training data. 
\end{enumerate}

\begin{es}
\emph{Let us work within the category BorelStoch
which is the subcategory of Stoch from
 Example \ref{es : markov category examples}. 
 We make this restriction because BorelStoch admits conditionals.
 In BorelStoch, we can view a morphism $f \colon A \to B$ as 
 defining a conditional distribution $p(b|a)$.}
 
  \emph{Let $X$ and $Y$ be Borel spaces. 
 Our training data consists of a list 
 $\mathbf{T} := [(x_1,y_1),\ldots,(x_n,y_n)]$
 of points in
 $\mathrm{List}(X \times Y)$. 
 To align with our notation 
 from \S \ref{sec : bayes learning algo}, 
 we write $X_T$ and $Y_T$ to be copies of $X$ and $Y$
 and endow $X_T \times Y_T$ with the empirical distribution 
 obtained from $\mathbf{T}$. 
 Our goal is to obtain an estimate of the conditional 
 probability $p(y_*|x_*)$ for general points $x_* \in X_*$
 and $y_* \in Y_*$.
 As for the training set, let $X_*$ and $Y_*$ be copies of $X$ and 
 $Y$ respectively which we use for inference.  
 }

 \emph{As outlined in \S \ref{sec : bayes learning algo} and 
 \S \ref{section : bayesian inference prelims}, we
 model the given data via a parametrized function of the form  
 $f \colon P \times X_T \to Y_T$ where $P \in \mathrm{BorelStoch}$. 
 We endow $P$ with a prior distribution i.e. a state 
 $\pi_P \colon I \to P$ where $I = \{*\}$ is the unit object.
 We must update the prior $p$ to get the posterior distribution. 
 This is accomplished via the Bayesian inversion 
 $f^\dag \colon Y_T \to P \times X_T$. 
 Note that $f^\dag$ is not unique. 
 By conditionalizing, we get a map 
 $X_T \times Y_T \to P$ in BorelStoch which defines 
 the posterior.}
 
 \emph{In this case, the Bayes predictive density 
  as described in \S \ref{sec : bayes learning algo}
  is given by}
  \[ p(y_* | x_*, T) = \int_P p(y_*|P,x_*,T)p(P|x_*,T).\]
  \emph{This is a consequence of how compositions 
  are defined in Stoch i.e. via the Chapman-Kolmogorov equation 
  cf. \S \ref{es : markov category examples}.}
      \end{es}

\section{Bayes updates} \label{section : bayes updates}
  
      While Section \ref{sec : bayes learning algo} describes the Bayes Learning algorithm, we 
    observe that it does not provide a mechanism by which we can 
    \emph{update} the prior on the parameter space. 
    Recall, via Bayesian inversion we obtain a
    channel 
    $(X_T,\pi_{X_T}) \otimes (Y_T,\pi_{Y_T}) \to (M,\pi_M)$.
    However, in practice, when working in a suitable
    sub-category of Stoch or in FinSet,
    we are given 
    a training set 
    \[ \mathbf{T} := \{(x_1,y_1), \ldots, (x_n,y_n)\}\]
    which we use to obtain the posterior
    distribution on $M$.
    Our goal in this section is to 
    translate this into the categorical framework we have developed so far. 
    
    In the case of Stoch, we can represent a data point 
    $(x,y) \in X_T \times Y_T$ as the product of a pair of morphisms
    $I \to X_T$ and $I \to Y_T$
    mapping $*$ to the 
    probability measures that 
    concentrate at the points $x$ and $y$ respectively. 
    We use $\delta_x$,$\delta_y$ respectively to denote these maps.
    Given a channel $c \colon X \times Y \to M$ in 
    Stoch, we define a state on $M$ via 
    the composition 
    \[I \to I \times I \xrightarrow{\delta_{x_1} \times \delta_{y_1}} X \times Y 
    \xrightarrow{c} M.\] 
      This effectively defines 
      the posterior on $M$ given the training 
      set $\mathbf{T}_1$ where 
      $\mathbf{T}_1 := \{(x_1,y_1)\}$.
      The natural question to ask in this setting is 
      how to sequentially update the posterior 
      to achieve 
     the required update over the entire training dataset 
     and if one can also achieve such an update all at once. 
    We outline a possible solution in what follows. 
    
    For the remainder of this section we work with a Markov category 
    $(\mathcal{C},I,\otimes)$ such that $\mathcal{C} = \mathrm{Kl}(\mathcal{P})$ where
    $\mathcal{P} \colon \mathcal{D} \to \mathcal{D}$ is a symmetric monoidal 
    monad on the symmetric monoidal category $(\mathcal{D},I_{\mathcal{D}},\ast)$. 
    Furthermore, we suppose that $\mathcal{C}$ admits conditionals. 
    As before, let $(\mathcal{M},J,\star)$ be a symmetric monoidal category 
    such that $\mathcal{C}$ is a symmetric monoidal $\mathcal{M}$-actegory. 
    
    We are given a model i.e. a morphism 
    $f \colon M \odot X \to Y$ in 
    $\mathcal{C}$
     where we think of $M \in \mathrm{Ob}(\mathcal{M})$ as the parameter space
      and $f$ models a true morphism $X \to Y$.
     We suppose as before that we are provided with a state
     $\pi_X \colon I \to X$ and 
     a prior $\pi_{M,0} \colon J \to M$.   
     Lastly, we abuse notation and write $M$ in place of $M \odot I$
     when necessary. Note a prior $\pi_{M} \colon J \to M$ 
     implies a prior $\pi_{M} \odot \mathrm{id} \colon J \odot I \to M \odot I$
     which we also refer to as $\pi_{M}$.

 \subsection{Sequential updates} \label{section : sequential updates}  
    Recall from 
    Definition \ref{def : disintegration and jointification} that
    the state $\pi_X$ and the channel $f$ give us
     a morphism 
     \begin{equation} \label{equn : f_joint}
     f_{\mathrm{joint}} \colon M \odot I \to X \otimes Y
     \end{equation}
 in $\mathcal{C}$. Since $\mathcal{C}$ admits conditionals, 
 it also admits Bayesian inversions. Hence, we get 
 a morphism 
 \[f_{\mathrm{joint}}^\dag \colon Z \to M \odot I \] with 
 respect to the prior state $\pi_{M,0}$ on $M$ where $Z := X \otimes Y$.
 
 Note that $f_{\mathrm{joint}}^\dag$ is not unique.
 In the previous section, we got around this issue 
 by working in the category $\mathrm{PS}(\mathcal{C})$. 
 However, if we want to update the state on $\mathcal{M}$ 
 sequentially then this corresponds to 
 sequentially updating objects in $\mathrm{PS}(\mathcal{M})$
 which will then require us to update the model $f$ or more precisely its image
 in $\mathrm{PS}(\mathcal{C})$.
 Instead, we introduce Definition \ref{def : unique inversion} 
 to ensure that the updated priors remain well defined. 
 
 Recall that 
 $\mathcal{C} = \mathrm{Kl}(\mathcal{P})$ where $\mathcal{P}$ is a monad
 on the symmetric monoidal category $(\mathcal{D},I_{\mathcal{D}},\ast)$.
 Since $\mathcal{P}$ is a monad, we have a family of 
 maps $\eta_X \colon X \to \mathcal{P}(X)$ for every $X \in \mathrm{Ob}(\mathcal{D})$. 
 Given a map $a \colon X \to Y$ in $\mathcal{D}$, let 
 $\eta(a)$ denote its image in $\mathcal{C}$ i.e. the composition 
 $X \xrightarrow{a} Y \xrightarrow{\eta_Y} \mathcal{P}(Y)$. 
 
 \begin{defi} \label{def : unique inversion}
  \emph{Let $f \colon X \to Y$ be a morphism in $\mathcal{C}$
  and $\pi_X \colon I \to X$ be a state on $X$. 
  Let $y \colon I_{\mathcal{D}} \to Y$ be a morphism 
  in $\mathcal{D}$. We say that the} Bayesian inverse of $f$ 
  is uniquely defined at $y$ \emph{if 
  for any morphisms $g,h \colon Y \to X$ in $\mathcal{C}$ such that 
  $g \sim_{\pi_X-a.s} h$ and $g$ is a Bayesian inversion of $f$ then 
  $\eta(y);g = \eta(y);h$.}
 \end{defi} 
 
 We refer to the morphism $y$ that appears in the definition above as an \emph{elementary point} of the 
 object $Y$. A precise definition is as follows. 
 
 \begin{defi} \label{def : elementary point}
 \emph{Let $Y$ be an object of $\mathcal{C}$.
  By an} elementary point of $Y$ \emph{we mean a morphism 
  $y \colon I_{\mathcal{D}} \to Y$
  in $\mathcal{D}$.
  We use $\delta_y$ to denote the image of $y$ in $\mathcal{C}$ i.e.
  $\delta_y := \eta(y)$.}
 \end{defi}
 
 \begin{es}
 \emph{We provide an example of Definition \ref{def : unique inversion}. 
 Recall from Example \ref{es : markov category examples}, the category 
 FinStoch. 
 Note that $\mathrm{FinStoch} = \mathrm{Kl}(\mathrm{Dist})$
 where $\mathrm{Dist} \colon \mathrm{FinSet} \to \mathrm{Finset}$ 
 is a symmetric monoidal monad on the symmetric monoidal category FinSet whose objects 
 are finite sets and morphisms are functions of sets. 
 Let $X$ be a finite set, $\pi_X \colon I \to X$ be a state on $X$ and 
 $f \colon X \to Y$ be a morphism in FinStoch. 
 By definition, $\pi_X$ corresponds to a probability distribution $p_X$
 on $X$ while $f$ defines a conditional distribution.
  Let $y_0 \colon I \to Y$ be a morphism in FinSet. 
  It follows that $y_0$ 
  is uniquely determined by a point in $Y$ which we abuse notation for and call $y_0$ as well. 
  The Bayesian inversion $g \colon Y \to X$ is defined by} 
  \[g(y)(x) = \frac{f(x)(y)p_X(x)}{\sum_{x' \in X} f(x')(y)p_X(x')}\]
  \emph{if $\sum_{x' \in X} f(x')(y)p_X(x') \neq 0$ and if $y$ is such that 
  $\sum_{x' \in X} f(x')(y)p_X(x') = 0$ then 
  $g(y)$ can be any probability distribution on $X$. 
   Thus we see in this situation that} 
   the Bayesian inversion $g$ is uniquely defined 
   at $y_0$ \emph{if and only if} 
   \[\sum_{x' \in X} f(x')(y_0)p_X(x') \neq 0.\]
   
   \emph{In the case of the category Stoch, things are more complicated since it 
   is not always true that a Bayesian inversion exists. However, in certain cases where 
   we are working with subspaces of $\mathbb{R}$ and
   both state and channel are defined using density functions then 
   a similar calculation as above can be performed 
   (cf.\cite[Example 3.9]{CJ2019disintegration}).}
   \end{es}   
  
 Let 
 $\mathbf{T} := [x_1 \otimes y_1, \ldots, x_n \otimes y_n]$
 be a list of $n$ elementary points of $X \otimes Y$ where for every $i$, 
 $z_i := x_i \otimes y_i$ satisfies a property to be specified below.  
 We begin with a prior $\pi_{M,0}$ on the parameter object $M$. 
  Let us suppose that we have 
  obtained the $i$-th sequential update i.e. 
  a state $\pi_{M,i}$ on $M$. 
  We define $\pi_{M,i+1}$ as follows. 
  By taking the Bayesian inversion of 
  $f_\mathrm{joint}$ with respect to $\pi_{M,i}$, 
  we get 
  \[f^{\dag}_{\mathrm{joint},i} \colon X \otimes Y \to M.\]
  The state $\pi_{M,i+1}$ on $X \otimes Y$ is given 
  by the composition 
  \[I \xrightarrow{\delta_{x_{i+1}} \otimes \delta_{y_{i+1}}} X \otimes Y \xrightarrow{f^{\dag}_{\mathrm{joint},i}} M\]
  and we suppose that the point $z_{i+1}$ is such that 
  $f^{\dag}_{\mathrm{joint},i}$ is unique at $z_{i+1}$.
  
  \begin{es}
    \emph{Let us demonstrate the sequential update procedure in the category 
    FinStoch acting on itself. 
    As above, we are given model $f \colon M \times X \to Y$ where $M,X$ and $Y$ are 
    finite sets and $f$ is a morphism in FinStoch. 
    This induces a function $f_{\mathrm{joint}} \colon M \to Z$ where 
    $Z := X \times Y$. 
    Let us assume 
    we are give a training set 
    \[ \mathbf{T} := [ (x_1,y_1), (x_2,y_2) ] \]
    and a prior state $\pi_{M,0}$ on $M$. In this context, this means
    a probability distribution on $M$. 
    Let $z_i := (x_i,y_i)$.}
    
   \begin{enumerate} \label{es : sequential updates}
    \item \emph{For $i = 1$, we have that for $m \in M$ and $z \in Z$,
    \[f^{\dag}_{\mathrm{joint},0}(z)(m)\pi_{Z,0}(z) = {f_{\mathrm{joint}}(m)(z)\pi_{M,0}(m)} \]
    where for $z \in Z$,}
    \[\pi_{Z,0}(z) = \sum_{m \in M} {f_{\mathrm{joint}}(m')(z)\pi_{M,0}(m')}.\]
   \emph{Recall our assumption that
   $\pi_{Z,0}(z_1) \neq 0$. 
   We update the prior by setting 
   $\pi_{M,1}(m) := f^{\dag}_{\mathrm{joint},0}(z_1)(m)$.}
    \item \emph{Likewise, for $i = 2$,}
      \[f^{\dag}_{\mathrm{joint},1}(z)(m)\pi_{Z,1}(z) = {f_{\mathrm{joint}}(m)(z)\pi_{M,1}(m)} \]
      \emph{where for $z \in Z$,}
    \[\pi_{Z,1}(z) = \sum_{m \in M} {f_{\mathrm{joint}}(m')(z)\pi_{M,1}(m')}.\]
    \emph{Expanding using (1),} 
     \[f^{\dag}_{\mathrm{joint},1}(z_2)(m) \propto {f_{\mathrm{joint}}(m)(z_2)f_{\mathrm{joint}}(m)(z_1)\pi_{M,0}(m)} \]
   \emph{We update the prior by setting 
   $\pi_2(m) := f^{\dag}_{\mathrm{joint},1}(z_2)(m)$.}  
    \end{enumerate}  
  \end{es}

  \subsection{Batch updates}
 
     To obtain an update of the prior all at once, we work with an object 
     built from $X \otimes Y$ but whose \emph{points} correspond 
     to datasets of a specified cardinality. 
     
     Let $n \in \mathbb{N}$ and 
     we set
     \[Z_n := \otimes^n (X \otimes Y).\] 
     The model $f$ induces a morphism 
    \[f_{\mathrm{joint}}^n \colon M \to Z_n \]
    defined as the composition 
    \[M \xrightarrow{\otimes^n \mathbf{copy}_M} \otimes^n M \xrightarrow{\otimes^n f_{\mathrm{joint}}} Z_n.\]
     As before, let
     $\mathbf{T} := [x_1 \otimes y_1, \ldots, x_n \otimes y_n]$
     be a list of $n$ elementary points of $X \otimes Y$. 
     The list $\mathbf{T}$ defines an elementary point 
     of $Z_n$. Indeed, we set 
     $z_\mathbf{T}$ to be the composition 
     \[I_{\mathcal{D}} \xrightarrow{\otimes^n \mathbf{copy}_{I_{\mathcal{D}}}} \otimes^n I_{\mathcal{D}} \xrightarrow{z_1 \otimes \ldots \otimes z_n} Z_n\] 
     where as before $z_i = x_i \otimes y_i$.
     We suppose $\mathbf{T}$ is
     such that 
     the Bayesian inversion 
     of $f_{\mathrm{joint}}^n$ is uniquely defined at 
     $z_\mathbf{T}$. 
     The batch update of 
     the prior $\pi_{M,0}$ with respect to 
     $\mathbf{T}$ is given by the 
     composition 
     \[I \xrightarrow{\delta_{z_\mathbf{T}}} Z \xrightarrow{(f_{\mathrm{joint}}^n)^{\dag}} M\]
   where
 $(f_{\mathrm{joint}}^n)^{\dag}$ is a Bayesian inversion of 
 $f_{\mathrm{joint}}^n$ with respect to the prior $\pi_{M,0}$. 
  Let $\pi_{M,\mathbf{T}} \colon I \to M$ denote this updated prior.

 We can ask the following question. Under what conditions, can we ensure 
 that $\pi_{M,n} = \pi_{M,\mathbf{T}}$ where $\pi_{M,n}$ is as defined 
 at the end of \S \ref{section : sequential updates}
 for $\mathbf{T}$. 
 
 \begin{es} \label{es : batch updates}
   \emph{Let us continue our discussion as in 
   Example \ref{es : sequential updates} using the category FinStoch.
    As in \ref{es : sequential updates}, 
    we are given a model $f \colon M \times X \to Y$ where $M,X$ and $Y$ are 
    finite sets and $f$ is a morphism in FinStoch, a training set}
    \[ \mathbf{T} := \{ (x_1,y_1), (x_2,y_2) \} \]
    \emph{and a prior state $\pi_{M,0}$ on $M$. 
    Let $z_i := (x_i,y_i)$ and $z_\mathbf{T} := (z_1,z_2)$.
     We import notation from \ref{es : sequential updates}.}
    
    \emph{It follows that}
    \[Z_2 = (X \times Y)^2\]
    \emph{and}
    \[f^2_\mathrm{joint} \colon M \to Z_2.\]
    \emph{Hence, we get that} 
    \[(f^2_{\mathrm{joint}})^{\dag} \colon Z_2 \to M.\]
    \emph{For $m \in M$ and $w \in Z_2$,} 
    \[(f^2_{\mathrm{joint}})^{\dag}(w)(m) \propto f^2_{\mathrm{joint}}(m)(w)\pi_{M,0}(m).\]
   \emph{We set} 
   \[ \pi_{M,\mathbf{T}}(m) := (f^2_{\mathrm{joint}})^{\dag}((z_1,z_2))(m)\]
    \emph{for $m \in M$
    and hence}
   \[ \pi_{M,\mathbf{T}}(m) \propto f_{\mathrm{joint}}(m)(z_2)f_{\mathrm{joint}}(m)(z_1)\pi_{M,0}(m)\]    
   \emph{since by definition,
   $f^2_{\mathrm{joint}}(m)((a,b)) = f_{\mathrm{joint}}(m)(a)f_{\mathrm{joint}}(m)(b).$}
 \end{es} 
   
  \begin{rem}
  \emph{Observe from examples \ref{es : sequential updates} and \ref{es : batch updates}
  that the sequential Bayes update and batch update coincide.
  This begs the following question. What conditions can we impose to
  relate the sequential and batch updates in the general setting 
  of the category $\mathcal{C}$ used throughout this section.}
  \end{rem}

  \bibliographystyle{plain}
\bibliography{library}

%

   \end{document}